\documentclass[12pt]{article}
\usepackage{geometry}
\usepackage{bbold}
\usepackage{graphicx}
\usepackage{tikz}
\usetikzlibrary{calc,patterns,decorations.pathmorphing,decorations.markings} 
\usepackage{amsthm}
\newtheorem{theorem}{Theorem}
\newtheorem{proposition}{Proposition}

\begin{document}
	
\title{Koopman Operator Family Spectrum \\ for Nonautonomous Systems - Part 1}
\author{Senka Ma\'{c}e\v{s}i\'{c}\footnote{Senka Ma\'{c}e\v{s}i\'{c} is with the Faculty of 		Engineering University of Rijeka, Croatia, senka.macesic@riteh.hr}, 
	    Nelida \v{C}rnjari\'{c}-\v{Z}ic\footnote{Nelida \v{C}rnjari\'{c}-\v{Z}ic is with the Faculty of Engineering University of Rijeka, Croatia, nelida@riteh.hr}, and 
	    Igor Mezi\'{c}\footnote{Igor Mezi\'{c} is with the Faculty of Mechanical Engineering and Mathematics, University of California, Santa Barbara, CA, 93106, USA mezic@engr.ucsb.edu}}

\maketitle

\begin{abstract}
For every non-autonomous system, there is the related family of Koopman operators $\mathcal{K}^{(t,t_0)}$, parameterized by the time pair $(t,t_0)$. In this paper we are investigating the time dependency of the spectral properties of the Koopman operator family in the linear non-autonomous case and we propose an algorithm for computation of its spectrum from observed data only. To build this algorithm we use the concept of the fundamental matrix of linear non-autonomous systems and some specific aspects of Arnoldi-like methods. In particular, we use Arnoldi-like methods on local data stencils, we exploit the information contained in the Krylov subspace projection error, and discover limitations in the application of Arnoldi-like methods to cases with continous time dependency. We present results of this data-driven algorithm on various linear non-autonomous systems, hybrid as well as continuous in time. In all the examples comparison with exact eigenvalues and eigenfunctions shows excellent performance of the proposed algorithm.         
\end{abstract}

\section{Introduction}

In recent years considerable attention has been paid to the analysis of dynamical systems behavior by analyzing the spectral properties of the associated Koopman operator. The Koopman operator was introduced in \cite{Koopman:1931}, as a composition operator acting on the space of observable functions. The crucial property of this operator is that it is linear on the space of observables.  The renewed interest for Koopman operator starts with the works \cite{mezicandbanaszuk:2004} and \cite{mezic:2005}, where authors studied the problem of decomposing evolution of a field from the perspective of the operator theory. They proved that under certain conditions the flow dynamics can be accurately decomposed into simpler structures that are based on projection to the eigenfunctions of Koopman operator associated with the dynamical evolution of the observables. Due to the linearity of the Koopman operator, this approach is applicable even if the dynamics is nonlinear. Thus, the dynamical evolution of the system can be described by using Koopman mode analysis, which consists of determining the eigenvalues, eigenfunctions and eigenmodes of the Koopman operator, so that the considered dynamical system can be represented by the corresponding Koopman Mode Decomposition (KMD). 
An overview of the spectral properties of Koopman operator and its application to the analysis of the fluid flow is given in \cite{rowleyetal:2009,Mezi2013,BudisicMezic_ApplKoop}.  

A variety of methods for determining the numerical approximation of the KMD has been developed. A general method for computing the Koopman modes, based on the rigorous theoretical results for the generalized Laplace transform, is known under the name Generalized Laplace Analysis (GLA) \cite{Mezi2013,BudisicMezic_ApplKoop}. It reduces to the Wiener's generalized harmonic analysis in the case when all the eigenvalues are on the unit circle \cite{Mauroy2012}. Another method that is closely related to the Koopman operator and is based on its spectral properties  is the Dynamic Mode Decomposition (DMD) method \cite{rowleyetal:2009}. Like GLA, DMD method belongs to the class of data driven algorithms, so that it can be applied to the time series of data, even if the underlying mathematical model is not known. The first DMD method for evaluating the Koopman mode and Koopman eigenvalues was the Arnoldi-like method based on the companion matrix \cite{rowleyetal:2009}. The more stable algorithm using the similar approach was based on the DMD decomposition and was proposed independently and with no relation to Koopman mode analysis by Schmid in \cite{schmid:2010,schmid:2011}. Tu et.al. provide in their paper \cite{Tu2014jcd} several alternative algorithms for evaluating DMD modes and eigenvalues and give comparison between them. They introduced the algorithm known under the name exact DMD. Williams et.al. introduced the extension of the DMD algorithm in \cite{williams2014data}, which they referred to as the Extended Dynamic Mode Decomposition (EDMD). This is an entirely data driven procedure for evaluating leading Koopman eigenfunctions, eigenvalues and modes from a data set of snapshot pairs and a dictionary of observables. Recently, Arbabi and Mezi\'{c} in \cite{ArbabiMezic2016} introduced a further extension of the DMD algorithm, Hankel DMD, based on the use of Hankel matrix instead of the companion matrix, for computation of the Koopman spectrum on single observable. They prove that the eigenvalues and eigenfunctions determined by the proposed Hankel-DMD method converge for the ergodic systems to the true eigenfunctions and eigenvalues of the infinite dimensional Koopman operator. The consequence of their results is that the convergence of the exact DMD \cite{Tu2014jcd} for the ergodic systems is obtained. 
Mentioned numerical algorithms for evaluating the spectral elements of Koopman operator have been successfully used to analyze different dynamical systems and flow configurations \cite{BudisicMezic_ApplKoop,susukiandmezic:2009,susukiandmezic:2012,GlazMezicFonoberovaLoire,
	eisenhoweretal:2010}. 

However, the Koopman operator analysis was almost exclusively applied to the autonomous systems. As far as we know, the Koopman operator framework was firstly extended to the non-autonomous dynamical systems in \cite{suranamezic}. They introduced a rigorous definition of the non-autonomous Koopman eigenvalues, eigenfunctions and modes, which are the building blocks of the non-autonomous Koopman mode decomposition used for describing the dynamic evolution of the flow governed by a non-autonomous system. This extension entails the time dependent eigenfunctions, eigenvalues and modes of the Koopman operator. They successfully applied the introduced extension to the linear periodic and quasi-periodic non-autonomous systems. The study of a non-autonomous dynamical system through the spectral properties of the corresponding Koopman operator is used in \cite{susukiandmezic:2009,susukiandmezic:2012} for analyzing the  power exchange deviation in the European grid disturbance. In these papers, the Arnoldi and Prony method for evaluating the Koopman mode decomposition were used. Recently, in \cite{KutzFuBrunton}, the multi resolution DMD (mrDMD) for decomposing data with multiple time scales has been proposed with the successful application to the non-stationary data.

The intention of this work is to apply the non-autonomous Koopman mode decomposition to linear non-autonomous systems. We use DMD algorithms, originally developed for the autonomous systems, to evaluate the time dependent eigenvalues, eigenfunctions and modes of the Koopman mode decomposition. We explore the limitation of the Arnoldi-like methods for their application to such systems. As a special case we consider the hybrid linear systems and develop a stable algorithm for evaluating the spectral decomposition by using the informations of the subspace projection error. Due to the special structure of the companion matrix used in the Arnoldi-like method, we show that, for the case of time-dependent linear systems, the appropriate choice of observables is necessary for determining the good approximations of non-autonomous Koopman eigenvalues. 

The paper is organized as follows. In Section \ref{sec:LNA-koop} the definition of the Koopman operator family for non-autonomous system is introduced and for the linear non-autonomous dynamical systems the relation with the fundamental matrix of the system is clarified in Theorem \ref{thm:kopp-fmatrix}. In Section \ref{sec:data-algorithm} we point out the issues that arise when Arnoldi-like methods are applied and we specify the error that arises from continuously changing underlying matrix in Theorem \ref{thm:NLA-Aerror}. Then we propose two algorithms that resolve these issues: Algorithm 1 for the hybrid dynamical systems, and Algorithm 2 for the dynamical systems with continuously changing underlying matrix. The advantages of the new algorithms are demonstrated in several numerical examples in Section \ref{sec:results}. 

\section{Koopman operator family of the linear \\ non-autonomous system}\label{sec:LNA-koop}

The linear non-autonomous system is a dynamical system governed by
\begin{equation}\label{eq:LNA-system}
\dot{\mathbf{x}}=\mathbf{A}(t)\mathbf{x}, 
\end{equation} 
where $\mathbf{x}=\mathbf{x}(t)$ is the $n$-dimensional state vector, and $\mathbf{A}=\mathbf{A}(t)$ is a given time-dependent matrix. 

For any dynamical system, the Koopman operator family $\mathcal{K}^{(t,t_0)}$ is defined with its action on the observables $f=f(\mathbf{x})$
\begin{equation}\label{eq:koop}
\mathcal{K}^{(t,t_0)}f(\mathbf{x}(t_0)) = f(\mathbf{x}(t)).
\end{equation} 
As usual in the Koopman operator framework, our main goal is to find non-autonomous Koopman eigenvalues $\lambda^{(t,t_0)}$ and eigenfunctions $\phi^{(t,t_0)}$ \cite{Mezi2013,suranamezic} defined by 
\begin{equation}\label{eq:koop-eigen-def}
\mathcal{K}^{(t,t_0)}\phi^{(t,t_0)} = e^{\lambda^{(t,t_0)}}\phi^{(t,t_0)}.
\end{equation} 


For the linear non-autonomous system (\ref{eq:LNA-system}), the fundamental matrix is defined as the matrix $\mathcal{M}(t,t_0)$ whose $i^{th}$ column is the solution of the equation (\ref{eq:LNA-system}) for the initial condition ${\bf x}(t_0)={\bf e}_i$, $i=1,...,n$. Then, the solution of (\ref{eq:LNA-system}) for the initial condition $\mathbf{x}(t_0)=\mathbf{x}_0$ can be written in the form 
\begin{equation}\label{eq:LNA-fmatrix}
\mathbf{x}(t) = \mathcal{M}(t,t_0) \mathbf{x}_0.
\end{equation} 

\begin{proposition}\label{thm:kopp-fmatrix}
Consider the fundamental matrix $\mathcal{M}(t,t_0)$ of the linear non-autonomous dynamical system (\ref{eq:LNA-system}). If 
\begin{equation}
(\mu_{i}^{(t,t_0)},{\bf w}_{i}^{(t,t_0)},{\bf v}_{i}^{(t,t_0)}), i=1,...,n 
\end{equation} 
are the eigenvalues, left and right eigenvectors of the fundamental matrix $\mathcal{M}(t,t_0)$,  then $\lambda_{i}^{(t,t_0)}$, $i=1,...,n$ such that 
\begin{equation}
\mu_{i}^{(t,t_0)}=e^{\lambda_{i}^{(t,t_0)}}, i=1,...,n
\end{equation} 
are the eigenvalues and 
\begin{equation}
\phi_{i}^{(t,t_0)}(\cdot)=\langle \cdot, {\bf w}_{i}^{(t,t_0)} \rangle, i=1,...,n
\end{equation} 
are the eigenfunctions of the Koopman operator ${\cal K}^{(t,t_{0})}$.
Furthermore, ${\bf v}_{i}^{(t,t_0)}$, $i=1,...,n$ are the Koopman modes of the full state observable and the following expansion is valid
\begin{equation}
\mathbf{x}(t)={\cal K}^{(t,t_{0})}(\mathbf{x}_0)=\sum_{i=1}^{n} e^{\lambda_{i}^{(t,t_0)}} \phi_{i}^{(t,t_0)}(\mathbf{x}_0) {\bf v}_{i}^{(t,t_0)}.
\end{equation}

\end{proposition}

\begin{proof}
The stated connection between Koppman operator ${\cal K}^{(t,t_{0})}$ and the fundamental matrix $\mathcal{M}(t,t_0)$ for linear non-autonomous dynamical system is easily derived using the linearity of (\ref{eq:LNA-system}) relative to $\mathbf{x}$.
\end{proof}

{\bf Example 1.}
Let as consider the scalar linear non-autonomous dynamical system
\begin{equation}\label{eq:ls-1d}
\dot{z} = a(t) z.
\end{equation} 
Since the solution of this system equals to 
$${\displaystyle z(t) = z(t_0) e^{ \int_{t_0}^{t} a(\tau) d\tau}},$$ 
it is quite easy to verify that 
$$\lambda_1^{(t,t_0)} =  \int_{t_0}^{t} a(\tau) d\tau \mbox{ and } \phi_1^{(t,t_0)} (z) = z$$
are the eigenvalue and the eigenfunction of the Koopman operator. Some of the other Koopman eigenvalues and eigenfunctions are 
$$\lambda_m^{(t,t_0)} = m \lambda_1^{(t,t_0)}\mbox{ and }\phi_m^{(t,t_0)} (z) = z^{m}, m=2,3,...$$ 
Several higher dimensional examples are examined in detail in Section \ref{sec:results}. \\

Additionally to what is stated in Theorem \ref{thm:kopp-fmatrix}, the fundamental matrix family has another property
\begin{equation}\label{eq:fmatrix-composition}
\mathcal{M}(t,t_0) = \mathcal{M}(t,t_1)\mathcal{M}(t_1,t_0)\mbox{, for all }t>t_1>t_0,
\end{equation} 
analogous to the Koopman operator family property
\begin{equation}
\mathcal{K}^{(t,t_0)} = \mathcal{K}^{(t,t_1)}\mathcal{K}^{(t_1,t_0)}\mbox{, for all }t>t_1>t_0.
\end{equation}

There are some important cases of linear non-autonomous systems for which fundamental matrix can be analytically obtained. We discuss these next.

The first case is the hybrid linear non-autonomous system, i.e. dynamical system (\ref{eq:LNA-system}) with the piecewise constant matrix 
\begin{equation}\label{eq:LNA-hybrid}
\mathbf{A}(t) = \sum_{l=0}^{\infty} \mathbf{A}_l\mathbb{1}_{\left[\right.T_{l},T_{l+1}\rangle}.
\end{equation}
Here $T_l$, $l=0,1,...$ is a sequence of time moments, and $\mathbf{A}_l$, $l=0,1,...$ is a sequence of constant matrices. In this case the fundamental matrix is given iteratively by
\begin{equation}\label{eq:LNA-hybrid-fmatrix}
\mathcal{M}(t,t_0) = e^{\mathbf{A}_{l(t)} (t-T_{l(t)})}\mathcal{M}(T_{l(t)},t_0)
\end{equation} 
where the index $l(t)$ is determined so that $t \in \left[\right.T_{l(t)},T_{l(t)+1}\rangle$.  

The second case is when the matrices $\mathbf{A}(t)$, $t>t_0$ have the same time independent eigenvectors, i.e. 
\begin{equation}\label{eq:NLA-comm}
\mathbf{A}(t)=\mathbf{R}\cdot\mathbf{\Lambda}_\mathbf{A}(t)\cdot\mathbf{R}^{-1} 
\end{equation} 
where $\mathbf{R}=\left( {\bf v}_{1} ... {\bf v}_{n} \right)$ is the matrix of the right eigenvectors, and $\mathbf{\Lambda}_\mathbf{A}(t)$ is the diagonal matrix with the corresponding eigenvalues on the diagonal. Then the fundamental matrix is given by 
\begin{equation}\label{eq:NLA-comm-fmatrix}
\mathcal{M}(t,t_0)=\mathbf{R}\cdot e^{\int_{t_0}^{t}\mathbf{\Lambda}_\mathbf{A}(\tau)d\tau} \cdot \mathbf{R}^{-1}.
\end{equation} 

In the general case, the fundamental matrix can be computed by some appropriate numerical method for the underlying system of ordinary differential equations.

\section{Data-driven algorithm for time dependent eigenvalues}\label{sec:data-algorithm}

Suppose that for some linear non-autonomous system (\ref{eq:LNA-system}) we have a sequence of snapshots of the full state observable 
\begin{equation}\label{eq:LNA-data}
\mathbf{x}_k=\mathbf{x}(t_k), k=0,1,...
\end{equation}
where $t_k=k\Delta t$, $k=0,1,...$. Our task is to compute spectrum of the Koopman operators $\mathcal{K}^{(t_k,t_0)}$, $k=0,1,...$ from these snapshots.
 
Due to the established connection with the spectrum of fundamental matrices, this task can be reduced to the computation of matrices $\mathcal{M}(t_k,t_0)$, $k=0,1,...$. From  (\ref{eq:fmatrix-composition}) we get 
\begin{equation}\label{eq:fmatrix-composition-dt}
\mathcal{M}(t_k,t_0) = \mathcal{M}(t_k,t_{k-1})\mathcal{M}(t_{k-1},t_0), k=1,2,...
\end{equation} 
This gives us the possibility to further reduce the problem to the approximate evaluation of local fundamental matrices $\mathcal{M}(t_k,t_{k-1})$, $k=1,2,...$.  

In order to do this, let us look at the local stencil of snapshots 
\begin{equation}\label{eq:LNA-local-stencil}
\mathbf{x}_{k-1}, \mathbf{x}_{k}, ..., \mathbf{x}_{k+s-1}
\end{equation}
with small $s\ge n$, which is the same for all stencils. To the local stencil (\ref{eq:LNA-local-stencil}) we can apply any of the Arnoldi-like methods and obtain a matrix $\mathbf{M}_{k,k-1}$ such that
\begin{equation}\label{eq:LNA-local-fmatrix}
\mathbf{x}_{k+j}\approx\mathbf{M}_{k,k-1} \mathbf{x}_{k+j-1}, j=0,1,...,s-1
\end{equation}
The approximation is obtained by the projection of $\mathbf{x}_{k+s-1}$ to the Krylov subspace spanned with $\mathbf{x}_{k-1}, \mathbf{x}_{k}, ..., \mathbf{x}_{k+s-2}$ 
\begin{equation}
c_0 \mathbf{x}_{k-1} + c_1 \mathbf{x}_{k} + \cdots c_{s-1} \mathbf{x}_{k+s-2} = \mathbf{x}_{k+s-1} + \mathbf{r} 
\end{equation}
under the condition that  
\begin{equation}\label{eq:project_cond}
\mathbf{r}\perp\mathbf{x}_{k-1},\mathbf{x}_{k},...,\mathbf{x}_{k+s-2}.
\end{equation}
The matrix representation of the projection operator in basis $\mathbf{x}_{k-1}, \mathbf{x}_{k}, ..., \mathbf{x}_{k+s-2}$ is given by the companion matrix
\begin{equation}\label{eq:companion-matrix}
\mathbf{C}=\left(\begin{array}{ccccc}
0 & 0 & \cdots  & 0 & c_0  \\ 
1 & 0 & \cdots  & 0 & c_1 \\ 
0 & 1 & \cdots  & 0 & c_2 \\ 
\vdots & \ddots & \ddots & \vdots  & \vdots  \\ 
0 & 0 & \cdots  & 1 & c_{s-1} 
\end{array} \right).
\end{equation}
The condition (\ref{eq:project_cond}) guaranties that the projection error
\begin{equation}\label{eq:AA-prj-err}
\|\mathbf{r}\|_2=\| \mathbf{x}_{k+s-1} - \left( c_0 \mathbf{x}_{k-1} + c_1 \mathbf{x}_{k} + \cdots c_{s-1} \mathbf{x}_{k+s-2} \right) \|_2  
\end{equation}
is minimal. 
Observe that the companion matrix (\ref{eq:companion-matrix}) is representation of a finite-dimensio-nal approximation of the Koopman operator $\mathcal{K}(t_k,t_{k-1})$ relative to the Krylov basis, while the matrix $\mathbf{M}_{k,k-1}$ is a representation of the same approximation, but relative to the original basis. After we find $\mathbf{M}_{k,k-1}$, $k=1,2,...$, we can construct the approximate fundamental matrix family
\begin{equation}
\mathbf{M}_{0,0} =\mathbf{I}\mbox{, } \mathbf{M}_{k,0} = \mathbf{M}_{k,k-1}\mathbf{M}_{k-1,0}, k=1,2,...
\end{equation} 

In the case when the fundamental matrix can be analyticialy computed, we can messure the error of the proposed algorithm using
\begin{equation}\label{eq:data-alg-err}
E_{k}= \left|\mathbf{M}_{k,0}-\mathcal{M}(t_k,t_0)\right|_{2}/\left|\mathcal{M}(t_k,t_0)\right|_{2}  
\end{equation}
Since at the end we apply Theorem \ref{thm:kopp-fmatrix} and use matrix family $\mathbf{M}_{k,0}$, $k=1,2,...$  to compute approximations of Koopman eigenvalues and eigenvectors, we can look at (\ref{eq:data-alg-err}) as an integrated error of the proposed approximation of the Koopman operator family spectrum.

\subsection{Hybrid linear non-autonomous system}\label{ss:hybrid-algorithm}

For the hybrid linear non-autonomous system (\ref{eq:LNA-hybrid}) the local fundamental matrix is given by
\begin{equation}\label{eq:LNA-hybrid-fmatrix-local}
\mathcal{M}(t_k,t_{k-1})=e^{\mathbf{A}_{l(k)}\Delta t}
\end{equation}
where $l(k)$ is such that $[t_{k-1},t_{k}]\subset\left[\right.T_{l(k)},T_{l(k)+1}\rangle.$ 

From the point of the local stencil (\ref{eq:LNA-local-stencil}), one possibility is that there is some $l(k)$ such that 
$$T_{l(k)+1}\in [t_{k-1},t_{k+s-1}].$$ 
Then, $\mathbf{M}_{k,k-1}$ is an attempt to approximate two very different matrices: $e^{\mathbf{A}_{l(k)}\Delta t}$ and $e^{\mathbf{A}_{l(k)+1}\Delta t}$. As shown in the examples in the next section, this results in a significant increase in the Krylov subspace projection error (\ref{eq:AA-prj-err}). Therefore, the projection error is a switch time indicator and we can use it to indentify all $T_l$, $l=0,1,...$. 

Suppose that $\Delta t$ is small enough, such that for each $l$, we can find $k(l)$ for which the following inclusion is valid
$$[t_{k(l)-1},t_{k(l)+s-1}]\subset\left[\right.T_{l},T_{l+1}\rangle.$$
This means that such local stencil (\ref{eq:LNA-local-stencil}) is completely produced by the action of matrix  
\begin{equation}\label{eq:LNA-hybrid-fmatrix-app}
\mathbf{M}_{k(l),k(l)-1}=e^{\mathbf{A}_{l}\Delta t},
\end{equation}
and we can use this to identify all $\mathbf{A}_l$, $l=0,1,...$.  

Therefore, in the hybrid linear non-autonomous system case, the proposed approach leads to a full identification of the system, and consequently also of its Koopman operator family. The resulting algorithm is summarized bellow. 
\\
\\
\begin{small}\label{alg:hybrid}
{\bf Algorithm 1 (for hybrid systems)}	

\begin{enumerate} 
\item Choose stencil size $s=n$ and maximal projection error $\epsilon>0$.
\item Apply Arnoldi-like method to local stencil of snapshots $\{\mathbf{x}_{k-1},\ldots,\mathbf{x}_{k+s-1}\}$,  to determine $\mathbf{M}_{k,k-1}$ and the projection error $\|r_k\|$ (\ref{eq:AA-prj-err}). 
\item If $\|r_k\| > \epsilon$, set $\mathbf{M}_{k,k-1}=\mathbf{M}_{k-1,k-2}$.
\item Compute  $\mathbf{M}_{k,0}=\mathbf{M}_{k,k-1}\mathbf{M}_{k-1,0}$
\item Compute dynamical system matrix eigenvalues from $\mathbf{M}_{k,k-1}$ and Koopman operator eigenvalues from $\mathbf{M}_{k,0}$.
\item Repeat steps 2-5, for all $k=1,2,...$. 
\end{enumerate} 
\end{small}

\subsection{Linear non-autonomous system with nonlinear time dependency}\label{ss:nonlinear-time}

\begin{theorem}\label{thm:NLA-Aerror}
Consider the dynamical system (\ref{eq:LNA-system}) with 
\begin{equation}\label{eq:LNA-oscillation}
\mathbf{A}(t) = 
\left(\begin{array}{rc}
\sigma(t) & \omega(t) \\ 
-\omega(t) & \sigma(t)
\end{array} \right)   
\end{equation}
where $\omega, \sigma\in C^2\left( \left[ t_0,\infty \right.\rangle \right)$, $\omega\ne 0$.
Let $\mu_i$, $i=1,2$ be complex conjugate eigenvalues of the companion matrix related to the Krylov subspace spanned with 
\begin{equation}\label{eq:Krylov3}
\mathbf{x}(t-\Delta t), \mathbf{x}(t), \mathbf{x}(t+\Delta t) .
\end{equation}
Then, the following relations hold
\begin{equation}\label{eq:AA-error-Re-lim}
\ln|\mu_i| = \left( \sigma(t) + \frac{\dot{\omega}(t)}{2\omega(t)} \right) \Delta t  + \mathcal {O}(\Delta t^2), 
\end{equation}
\begin{equation}\label{eq:AA-error-Im-lim}
\mbox{Arg}(\mu_i) = \omega(t)\Delta t\sqrt{1-\frac{\dot{\sigma}(t)}{\omega(t)^2}}+ \mathcal {O}(\Delta t^2). 
\end{equation}
for every $t\in\left[ t_0,\infty \right.\rangle$.
\end{theorem}

\begin{proof}
Since 
\begin{equation}
\mathbf{A}_1 \mathbf{A}_2 =
\left(\begin{array}{rc}
\sigma_1\sigma_2-\omega_1\omega_2 & \sigma_1\omega_2+\omega_1\sigma_2 \\ 
-(\sigma_1\omega_2+\omega_1\sigma_2) & \sigma_1\sigma_2-\omega_1\omega_2
\end{array} \right)   
\end{equation}
where $\mathbf{A}_i=\mathbf{A}(t_i)$, $t_i\ge t_0$, $i=1,2$, i.e. matrices (\ref{eq:LNA-oscillation}) are commutative,  we can apply (\ref{eq:NLA-comm-fmatrix}) and find the fundamental matrix
\begin{equation}\label{eq:osci-fmatrix}
\mathcal{M}(t,t_0)=
e^{\alpha(t,t_0)}
\left(\begin{array}{rc}
 \cos\beta(t,t_0) & \sin\beta(t,t_0) \\ 
-\sin\beta(t,t_0) & \cos\beta(t,t_0) 
\end{array} \right),   
\end{equation}
where 
\begin{equation}\label{eq:osci-ab}
\alpha(t,t_0)=\int_{t_0}^{t}\sigma(\tau)d\tau \mbox{, and  } \beta(t,t_0)=\int_{t_0}^{t}\omega(\tau)d\tau.
\end{equation}  

For the initial condition written in the form
$$\mathbf{x}(t_0)=
e^{\alpha_0}
\left(\begin{array}{c}
 \cos\beta_0 \\ 
 \sin\beta_0 
\end{array} \right),$$   
the solution is 
\begin{equation}
\mathbf{x}(t)=
\mathcal{M}(t,t_0)\mathbf{x}(t_0)=
e^{\alpha(t,t_0)+\alpha_0}
\left(\begin{array}{r}
 \cos(\beta(t,t_0)-\beta_0) \\ 
-\sin(\beta(t,t_0)-\beta_0) 
\end{array} \right).   
\end{equation}
Now, for a chosen $t\in\left[ t_0,\infty \right.\rangle$ and $\Delta t>0$ we look at the Krylov subspace spanned with $\mathbf{x}_{-}=\mathbf{x}(t-\Delta t), \mathbf{x}=\mathbf{x}(t), \mathbf{x}_{+}=\mathbf{x}(t+\Delta t)$ 
and compute the related companion matrix, i.e. we must find $c_0,c_1$ such that
$$
c_0 \mathbf{x}_{-} + c_1 \mathbf{x} = \mathbf{x}_{+} + \mathbf{r} \mbox{, and  } \mathbf{r}\perp\mathbf{x}_{-},\mathbf{x}.
$$
The solution is
\begin{equation}
c_0 = -e^{\alpha_{-}+\alpha_{+}}\frac{\sin\beta_{+}}{\sin\beta_{-}},
\end{equation}
\begin{equation}
c_1 = e^{\alpha_{+}}\frac{\sin(\beta_{-}+\beta_{+})}{\sin\beta_{-}},
\end{equation}
where $\alpha_{\pm}=\pm\alpha(t\pm\Delta t,t)$ and $\beta_{\pm}=\pm\beta(t\pm\Delta t,t)$. The eigenvalues $\mu_i$ are roots of the equation $\mu^2-c_1\mu-c_0=0$ so they satisfy
\begin{equation}\label{eq:Arnoldi-root-Re}
\ln|\mu_i| = \frac{1}{2}\ln(-c_0),
\end{equation}
\begin{equation}\label{eq:Arnoldi-root-Im}
\tan(\mbox{Arg}(\mu_i))=\sqrt{\left(\frac{2\sqrt{-c_0}}{c_1} \right)^2-1}.
\end{equation}

Since $\omega=\omega(\tau)$ and $\sigma=\sigma(\tau)$ can be approximated with Taylor polynomials of first degree in the neighborhood of $t$, with integration we get
\begin{equation}\label{eq:AA-alphapm}
\alpha_{\pm} = \sigma(t)\Delta t \pm \frac{1}{2}\dot{\sigma}(t)\Delta t^2 + \mathcal {O}(\Delta t^3),
\end{equation}
\begin{equation}\label{eq:AA-betapm}
\beta_{\pm}  = \omega(t)\Delta t \pm \frac{1}{2}\dot{\omega}(t)\Delta t^2 + \mathcal {O}(\Delta t^3).
\end{equation}

By applying (\ref{eq:AA-alphapm}) and (\ref{eq:AA-betapm}) to (\ref{eq:Arnoldi-root-Re}) we get
\begin{equation}\label{eq:AA-error-Re}
\ln|\mu_i| = \sigma(t)\Delta t + \frac{1}{2}\ln\left(\frac{\sin\beta_{+}}{\sin\beta_{-}} \right)+
\mathcal{O}(\Delta t^2) 
\end{equation}
Further computations give us
\begin{equation}
\lim\limits_{\Delta t\rightarrow 0} \ln\left(\frac{\sin\beta_{+}}{\sin\beta_{-}} \right)=0,
\end{equation}
but
\begin{equation}
\lim\limits_{\Delta t\rightarrow 0} \frac{1}{\Delta t}\ln\left(\frac{\sin\beta_{+}}{\sin\beta_{-}} \right)=\frac{\dot{\omega}(t)}{\omega(t)}.
\end{equation}
Therefore,  (\ref{eq:AA-error-Re-lim}) is valid.

For the imaginary part, we apply (\ref{eq:AA-alphapm}) and (\ref{eq:AA-betapm}) to (\ref{eq:Arnoldi-root-Im}) and get
\begin{equation}\label{eq:AA-error-Im}
\tan(\mbox{Arg}(\mu_i))=\tan(\omega(t)\Delta t)\sqrt{1+\frac{e^{-\dot{\sigma}(t)\Delta t^2}-1} {\sin^2(\omega(t)\Delta t)}} + \mathcal {O}(\Delta t^2).
\end{equation}
Computations give us
\begin{equation}
\lim\limits_{\Delta t\rightarrow 0} \frac{e^{-\dot{\sigma}(t)\Delta t^2}-1}{\sin^2(\omega(t)\Delta t)} = 
\frac{-\dot{\sigma}(t)}{\omega(t)^2},
\end{equation}
Therefore, (\ref{eq:AA-error-Im-lim}) is also valid. 
\end{proof}

From the Arnoldi-like method applied to the Krylov subspace (\ref{eq:Krylov3}) we expect to obtain the approximation of the eigenvalues that correspond to the action of the  Koopman operator on that subspace. Thus, due to its close relation with the fundamental matrix we should expect to get the eigenvalues of the fundamental matrix $\mathcal{M}(t,t-\Delta t) \approx e^{\mathbf{A}(t) \Delta t}$, which are approximately equal to $e^{\sigma(t) \Delta t \pm i \omega(t) \Delta t}$. However, the consequence of the Theorem \ref{thm:NLA-Aerror} is that any Arnoldi-like method applied to the dynamical system with the underlying matrix (\ref{eq:LNA-oscillation}) mixes the values of the real and imaginary parts of the local Koopman eigenvalues and produces an error that doesn't vanish with $\Delta t\rightarrow 0$.  This issue is visible in examples presented in Section \ref{sec:results}. Since form of the matrix (\ref{eq:LNA-oscillation}) is the form of the Jordan block belonging to any couple of complex conjugate eigenvalues, the discovered issue goes far beyond the examined two-dimensional dynamical system. 

The appropriate way to handle this issue in the Koopman framework is to redefine the observables. Application of (\ref{eq:koop-eigen-def}) to the linear non-autonomous system (\ref{eq:LNA-oscillation}) gives eigenvalues and eigenfunctions of the Koopman operator family $\mathcal{K}^{(t,t_0)}$ that can be computed from the spectrum of matrix $\mathbf{A}(t)$ 
\begin{equation}
\lambda_{\pm}^{(t,t_0)}=\alpha(t,t_0)\pm i\beta(t,t_0), \phi_{\pm}^{(t,t_0)}=x_1\mp ix_2.
\end{equation} 
However, this choice of observables doesn't solve the Arnoldi-like method issue because it doesn't decouple real and imaginary part of the eigenvalues. 

We can obtain the desired decoupling by using observables defined by 
\begin{equation}\label{eq:good-observables}
 u_1 = \sqrt{x_1^2+x_2^2} \mbox{ and } 
 u_2 = (x_1 + ix_2)/u_1.
\end{equation} 
These observables are also eigenfunctions of the same Koopman operator family $\mathcal{K}^{(t,t_0)}$ related to eigenvalues 
\begin{equation}
\lambda_1^{(t,t_0)}=\alpha(t,t_0) \mbox{ and } \lambda_2^{(t,t_0)}=-i\beta(t,t_0)
\end{equation} 
respectively.
It is easy to see that vector of these observables $\mathbf{u}=(u_1,u_2)^T$ satisfies
\begin{equation} \label{eq:LNA-osci-redef}
\dot{\mathbf{u}}=\mathbf{\tilde{A}}(t)\mathbf{u}, 
\end{equation}
with diagonal time dependent matrix $\mathbf{\tilde{A}}(t)$ whose diagonal elements are
\begin{equation}
\sigma(t) \mbox{ and } -i\omega(t).
\end{equation} 
The fundamental matrix $\tilde{\mathcal{M}}(t,t_0)$ for (\ref{eq:LNA-osci-redef}) is also diagonal, with diagonal elements
\begin{equation}
e^{\alpha(t,t_0)} \mbox{ and } e^{-i\beta(t,t_0)}.
\end{equation} 
With this choice of observables the proposed algorithm works well, and we get accurate identification of eigenvalues, as it can be seen in the examples in Section \ref{sec:results}. At the end, by using relations
\begin{equation}
x_1=u_1(u_2+\bar{u_2})/2 \mbox{ and } x_2=u_1(u_2-\bar{u_2})/2i
\end{equation} 
we can reconstruct the full state observables. 

The algorithm is summarized bellow. 
\\
\\
\begin{small}\label{alg:continuous}
{\bf Algorithm 2 (for continuous time dependency systems)}	
	
\begin{enumerate} 
\item Chose a set of observables $\mathbf{u}=(u_1,...,u_m)^T$ appropriate for the considered system. 
\item Apply Arnoldi-like method to local stencil of two snapshots $\{u_i(t_{k-1}),u_i(t_{k})\}$, separately for each $i=1,...,m$, and then determine $\mathbf{\tilde{M}}_{k,k-1}$. 
\item Compute  $\mathbf{\tilde{M}}_{k,0}=\mathbf{\tilde{M}}_{k,k-1}\mathbf{\tilde{M}}_{k-1,0}$
\item Compute dynamical system matrix eigenvalues from $\mathbf{\tilde{M}}_{k,k-1}$ and Koopman operator eigenvalues from $\mathbf{\tilde{M}}_{k,0}$.
\item Repeat steps 2-4, for all $k=1,2,...$. 
\end{enumerate} 
\end{small}

This approach leads us to the same good observables / coordinates that were pointed out in \cite{BudisicMezic_ApplKoop} through theoretical investigation.  
Notice that if $r(\mathbf{x})$ and $\Theta(\mathbf{x})$ are polar coordinates of the state vector $\mathbf{x}$ as defined in \cite{BudisicMezic_ApplKoop}, then $u_1 = r(\mathbf{x})$ and $u_2 = e^{ i \Theta(\mathbf{x})}$. Thus, we can conclude that the eigenfunctions $r(\mathbf{x})$ and $e^{i \Theta(\mathbf{x})}$ provide us with a good coordinate system for studying dynamics of the considered system. Moreover, as it was shown in \cite{BudisicMezic_ApplKoop}, the initial system defined with (\ref{eq:LNA-system}) and (\ref{eq:LNA-oscillation}) and the obtained system (\ref{eq:LNA-osci-redef}) are 
topologically conjugate to each other  through the nonlinear conjugate transformation 
$m : \mathbf{R}^{2} \rightarrow \mathbf{R}^{+} \times S^{1}$ given by
\begin{equation}\label{eq:ConjMap}
m(\mathbf{x})=\left(r(\mathbf{x}),e^{i \Theta(\mathbf{x})} \right).
\end{equation}

\section{Numerical results}\label{sec:results}

In this section first we demonstrate the application of Algorithm 1 on the hybrid linear non-autonomous systems (test examples \ref{ss:h2f}, \ref{ss:hdd}, \ref{ss:hco3}, and \ref{ss:mcm}). Then, on the linear non-autonomous systems with continuously changing underlying matrix (test examples \ref{ss:c2f}, \ref{ss:cdd}, and \ref{ss:cco3}), we show performance of Algorithm 2. Effects of the change (discontinuous and continuous) in the imaginary part of the underlying matrix eigenvalues are examined in test examples \ref{ss:h2f} and \ref{ss:c2f}; while changes in the real part are examined in test examples \ref{ss:hdd} and \ref{ss:cdd}. Examples \ref{ss:hco3} and \ref{ss:cco3} are introduced to test new algorithms on higher dimensional dynamical systems with all state variables strongly coupled. Test example \ref{ss:mcm} is a successful application of Algorithm~1 to multicompartment models with time delay \cite{Svenkeson}, which is often referenced in medicine and pharmacotherapy. 

All the examples are chosen so that exact solutions and exact Koopman eigendecompositions can be computed. Also, if not stated otherwise, a sequence of snapshots for testing data-driven algorithms is provided using time step $\Delta t=0.01$.  

The SVD enhanced DMD algorithm \cite{schmid:2010} is the Arnoldi-like method used in step 2 of both new algorithms. Also, new algorithms are compared with the same method applied on the moving stencil. The Standard DMD in all plot legends denotes the DMD algorithm from  \cite{schmid:2010}.   
Finally, in all the plots in which exact and numerical values overlap, the exact values are very slightly offset to assure both sets of data are visible.

\subsection{Switching frequency}\label{ss:h2f}

An oscillator with the switching frequency has governing equations of the form (\ref{eq:LNA-system}) with the underlying matrix (\ref{eq:LNA-hybrid}) where 
\begin{equation}\label{eq:ex_2f}
\mathbf{A}_l = \left\{\begin{array}{ll}
\left(\begin{array}{cc}
0 & 1 \\ 
-\omega_1^2 & 0
\end{array} \right), & l=0,2,4,... \\ 
\\
\left(\begin{array}{cc}
0 & 1 \\ 
-\omega_2^2 & 0
\end{array} \right), & l=1,3,5,... 
\end{array} \right. 
\end{equation}

The eigenvalues of the underlying matrices are $\pm \omega_1 i$ and $\pm \omega_2 i$, and the matrices are non-commutative. For the frequency values $\omega_1=2$, $\omega_2=1$, and switching times $T_l=l$, $l=1,2,...$ the oscillator is unstable. The exact solution is given by (\ref{eq:LNA-hybrid-fmatrix}). 

In Fig.\ref{fig:h2f-01}(b) we show the exact fundamental matrix eigenvalues for the time interval $[0,5]$, and we observe eigenvalues that are not on the unit circle. This time interval is chosen since the values of the eigenvalues off the unit circle grow even more with time, and the unit circle on such plot would not be visible. In Fig.\ref{fig:h2f-02} the eigenvalues of the underlying dynamical system matrix are plotted and these are not revealing the instability. However, when the time-dependent Koopman operator eigenvalues are plotted in Fig.\ref{fig:h2f-03}, we clearly observe the process of the amplitude growth. The negative real parts of Koopman eigenvalues belong to the particular solution that decays, which is added to the particular solution that has growing real parts of eigenvalues; so the solution is unstable. On these figures exact values and values computed with Algorithm 1 completely overlap. However, if the dynamical system matrix eigenvalues are computed with the standard DMD algorithm on moving stencils (as in Fig.\ref{fig:h2f-02-err}), incorrect values appear at every switch. This causes a significant error in Koopman operator eigenvalues also (as in Fig.\ref{fig:h2f-03-err}). In Fig.\ref{fig:h2f-02-err}(c) we observe how the Krylov subspace projection error has large values at every switch, which is the key point for Algorithm 1.

\begin{figure}[!htb]
	\centering
	\includegraphics[width=0.9\linewidth]{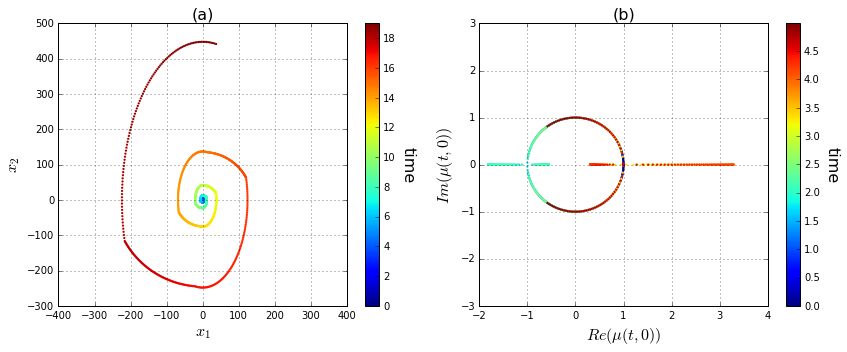}
	\caption{Exact solution starting with the inital condition (1,1) in the state space (a), and exact fundamental matrix eigenvalues over time interval $[0,5]$ (b) for the dynamical system with switching frequency (\ref{eq:ex_2f}).}
	\label{fig:h2f-01}
\end{figure}
	
\begin{figure}[!htb]
	\centering
	\includegraphics[width=0.9\linewidth]{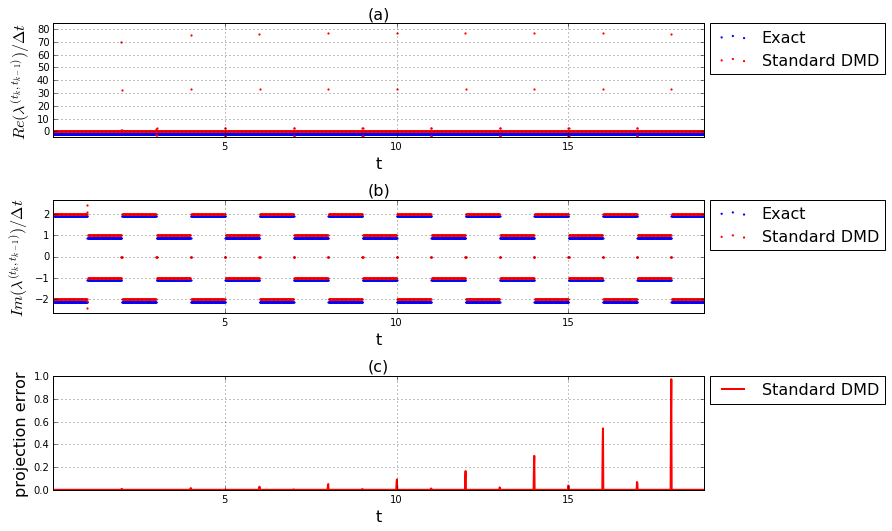}
	\caption{Dynamical system matrix eigenvalues and the projection error (\ref{eq:AA-prj-err}) for the dynamical system with switching frequency (\ref{eq:ex_2f}), computed with standard DMD.}
	\label{fig:h2f-02-err}
	
	\includegraphics[width=0.9\linewidth]{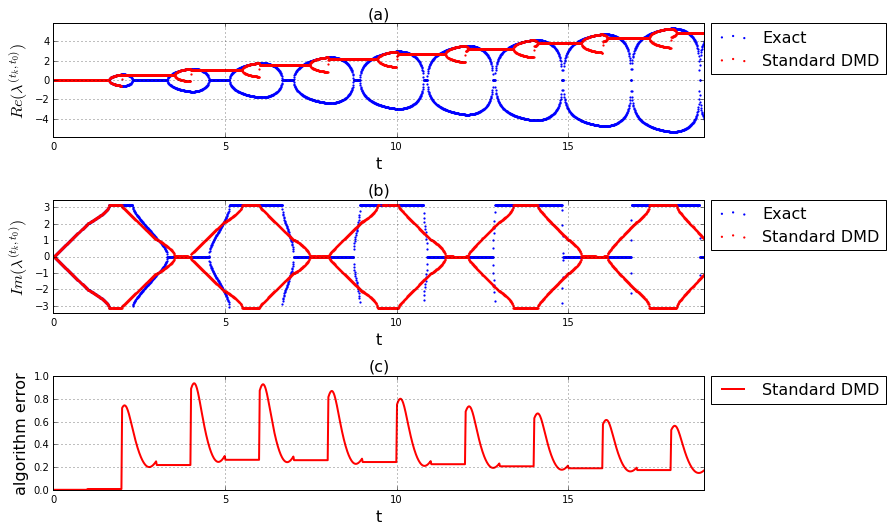}
	\caption{Koopman operator eigenvalues and the algorithm error (\ref{eq:data-alg-err}) for the dynamical system with switching frequency (\ref{eq:ex_2f}), computed with standard DMD.}
	\label{fig:h2f-03-err}
\end{figure}

\begin{figure}[!htb]
	\centering
	\includegraphics[width=0.9\linewidth]{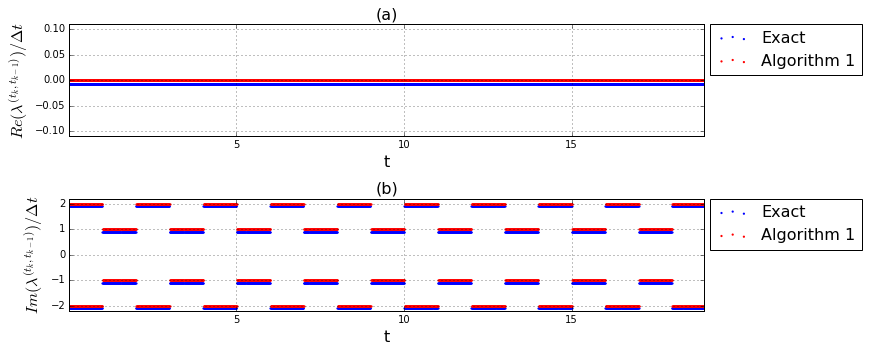}
	\caption{Dynamical system matrix eigenvalues for the dynamical system with switching frequency (\ref{eq:ex_2f}), computed with Algorithm 1. (Exact values are offset to improve visibility.)}
	\label{fig:h2f-02}
	
	\includegraphics[width=0.9\linewidth]{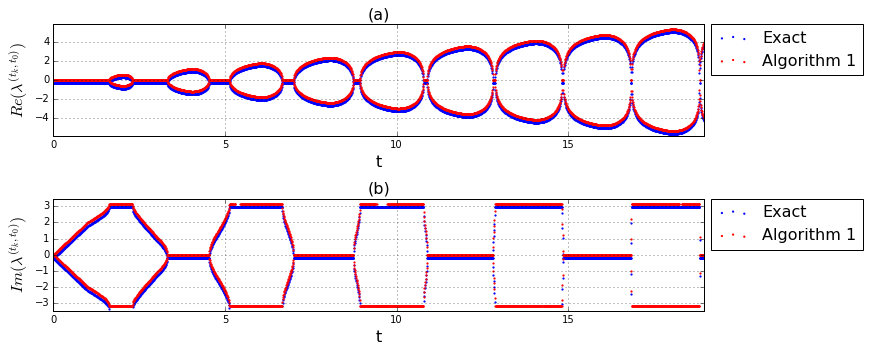}
	\caption{Koopman operator eigenvalues  for the dynamical system with switching frequency (\ref{eq:ex_2f}), computed with Algorithm 1. (Exact values are offset to improve visibility.)}
	\label{fig:h2f-03}
\end{figure}

\subsection{Switching damped-driven behavior}\label{ss:hdd}

In this example we consider an oscillator with the switching damped-driven behavior, i.e. with governing equations (\ref{eq:LNA-system}), the underlying matrix (\ref{eq:LNA-hybrid}) and 

\begin{equation}\label{eq:ex_dd}
\mathbf{A}_l = \left\{\begin{array}{ll}
\left(\begin{array}{cc}
\sigma_1 & 1 \\ 
-4 & \sigma_1
\end{array} \right), & l=0,2,4,... \\ 
\\
\left(\begin{array}{cc}
\sigma_2 & 1 \\ 
-4 & \sigma_2
\end{array} \right), & l=1,3,5,... 
\end{array} \right. 
\end{equation}

The eigenvalues of the underlying matrices are $\sigma_1\pm 2i$ and $\sigma_2\pm 2i$. For the real part $\sigma_1=1$, $\sigma_2=-1$, and switching times $T_l=T_{l-1}+l/2$, $l=1,2,...$, the matrices are commutative. Therefore, we can solve the system analytically and obtain the fundamental matrix
\begin{equation}\label{eq:dd_fmatrix}
{\cal M}(t,0)=
\left(\begin{array}{cc}
e^{\alpha(t,0)}\cos(2t) & \frac{1}{2}e^{\alpha(t,0)}\sin(2t) \\ 
-2 e^{\alpha(t,0)}\sin(2t) & e^{\alpha(t,0)}\cos(2t) \end{array} \right),   
\end{equation}
and its eigenvalues $\mu(t,0)=e^{\alpha(t,0)\pm 2ti}$. Here 
\begin{equation}\label{eq:dd_fmatrix_eigens}
\alpha(t,0)=\int_{0}^{t}(-1)^{l}\mathbb{1}_{\left[\right.T_{l},T_{l+1}\rangle} dt.
\end{equation}

In this case the real part $\sigma$ is switching between two values of which one causes driven behavior and the other damped behavior. We show the exact solution in Fig.\ref{fig:hdd-01}. In Fig.\ref{fig:hdd-02} we observe correct eigenvalues of the underlying matrix. The time-dependent Koopman eigenvalues in Fig.\ref{fig:hdd-03} clearly show oscillation (plot(b)), and damping-driving switching at correct switching times (plot (a)). The results obtained with Algorithm 1 match the exact ones, up to the machine round-off error.

\begin{figure}[!htb]
	\centering
	\includegraphics[width=0.9\linewidth]{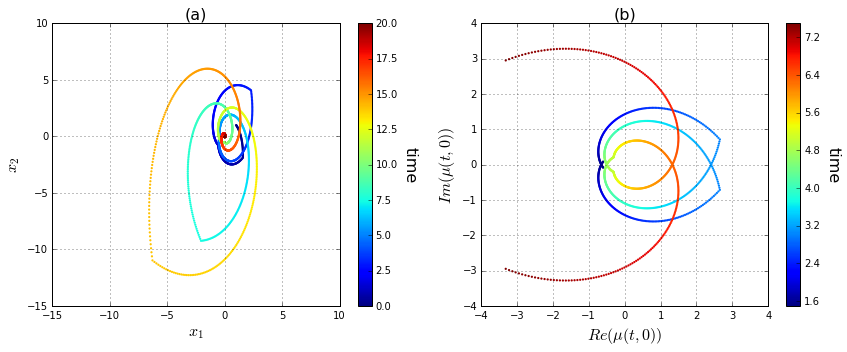}
	\caption{Exact solution starting with the inital condition (1,1) in the state space (a), and exact fundamental matrix eigenvalues (b), for the dynamical system with switching damped-driven behavior (\ref{eq:ex_dd}).}
	\label{fig:hdd-01}
\end{figure}

\begin{figure}[!htb]
	\centering
	\includegraphics[width=0.9\linewidth]{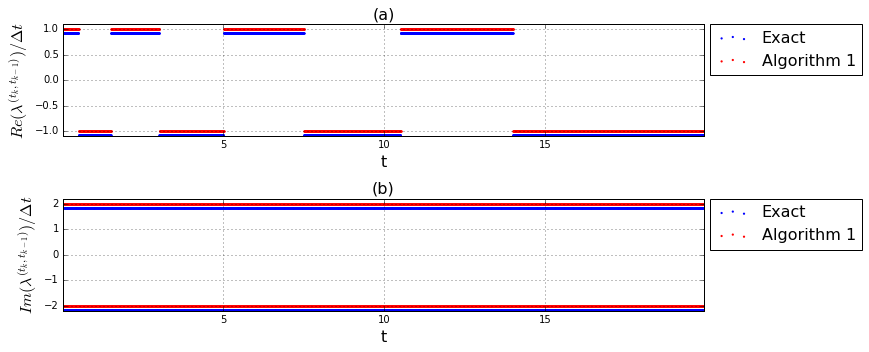}
	\caption{Dynamical system matrix eigenvalues for the dynamical system with switching damped-driven behavior (\ref{eq:ex_dd}), computed with Algorithm 1. (Exact values are offset to improve visibility.)}
	\label{fig:hdd-02}

	\includegraphics[width=0.9\linewidth]{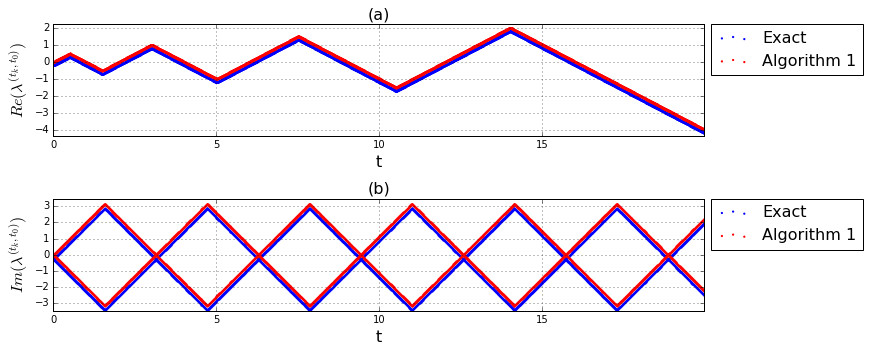}
	\caption{Koopman operator eigenvalues for the dynamical system with switching damped-driven behavior (\ref{eq:ex_dd}), computed with Algorithm 1. (Exact values are offset to improve visibility.)}
	\label{fig:hdd-03}
\end{figure}

\subsection{Coupled oscillators with switching frequency}\label{ss:hco3}

Here we consider a two degrees of freedom oscillator with masses $m_1$ and $m_2$, spring elasticities $k_1$, $k_2$, and $k_3$ as shown in Fig.\ref{fig:mass_spring} \cite{Veselic}.

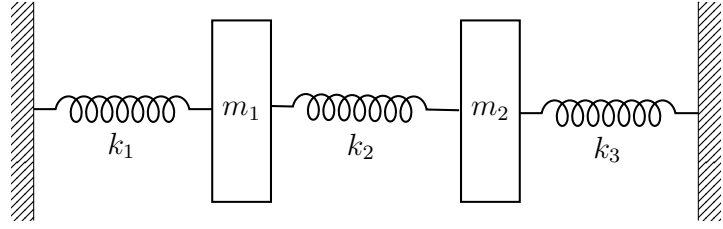
\begin{figure}
	\centering
		\resizebox{0.6\linewidth}{!}{
		\begin{tikzpicture}[mass/.style={draw,outer sep=0pt,thick}]
		\tikzstyle{spring}=[thick,decorate,decoration={coil,aspect=0.7,amplitude=5,pre length=0.3cm,post length=0.3cm,segment length=6}]
		\tikzstyle{ground}=[fill,pattern=north east lines,draw=none,minimum width=0.75cm,minimum height=0.3cm]
		\node [mass] (M1) [minimum width=0.5cm, minimum height=2.5cm] {$m_1$};
		\node (wall) [ground, rotate=-90, minimum width=3cm,yshift=-3cm] {};
		\draw (wall.north east) -- (wall.north west);
		\draw [spring] (wall.100) -- ($(M1.north west)!(wall.100)!(M1.south west)$) 
		node[below = 0.5 cm, left = 0.9 cm]  {$k_1$};
		\node[at={($(M1.east)+(3cm,0)$)}] [mass] (M2) [minimum width=0.5cm, minimum height=2.5cm] {$m_2$};
		\draw [spring] ($(M1.south east)!(M1.170)!(M1.north east)$) -- ($(M1.north east)!(M2.170)!(M2.west)$) 
		node[below = 0.5 cm, left = 1 cm] {$k_2$};
		\node[at={($(M1.east)+(3cm,0)$)}] (wall2) [ground, rotate=90, minimum width=3cm,yshift=-3cm] {};
		\draw (wall2.north east) -- (wall2.north west);
		\draw [spring] ($(M2.north east)!(wall2.100)!(M2.south east)$) -- (wall2.100) 
		node[below = 0.5 cm, left = 0.9 cm]  {$k_3$}; 
		\end{tikzpicture}
	}
	\caption{Mass spring oscillator}
	\label{fig:mass_spring}
\end{figure}

For this system of coupled oscillators the Newton's law gives
\begin{equation}
m_1 \ddot{x}_1=-k_1 x_1-k_2(x_1-x_2)
\end{equation}  
\begin{equation}
m_2 \ddot{x}_2=-k_2(x_2-x_1)-k_3 x_2
\end{equation}  
where $x_1$ and $x_2$ are mass displacements from the equilibrium position. 
If we add variables $x_3=\dot{x}_1$ and $x_4=\dot{x}_2$ we obtain (\ref{eq:LNA-system}) with the underlying matrix
\begin{equation}\label{eq:hco3-matrix}
\mathbf{A} = 
\left(\begin{array}{cccc}
0 & 0 & 1 & 0 \\ 
0 & 0 & 0 & 1 \\ 
-\frac{k_1+k_2}{m_1} & \frac{k_2}{m_1} & 0 & 0 \\ 
\frac{k_2}{m_2} & -\frac{k_2+k_3}{m_2} & 0 & 0  
\end{array} \right).   
\end{equation}
If we solve the generalized eigenvalue problem
\begin{equation}\label{eq:co3-geneig-1}
\det\left(\mathbf{K}-\nu_j\mathbf{M}\right)=0, 
\end{equation}
$j=1,2$ for matrices
\begin{equation}\label{eq:co3-geneig-2}
\mathbf{K}=
\left(
\begin{array}{cc}
k_1+k_2 & -k_2 \\
-k_2 & k_2+k_3 
\end{array}
\right) \mbox{ and }
\mathbf{M}=
\left(
\begin{array}{cc}
m_1 & 0 \\
0 & m_2 
\end{array}
\right)
\end{equation}   
then 
\begin{equation}
\pm i\omega_j = \pm i\sqrt{\nu_j}, \quad j=1,2
\end{equation}
are the four eigenvalues of (\ref{eq:hco3-matrix}). 

Now we suppose there is a sequence of time moments 
$T_l=l$, $l=0,1,...$ at which elasticity coefficients change value, i.e.
\begin{equation}
k_j(t)=k_j^{(l)}, \mbox{ for } t\in[T_l,T_{l+1}\rangle, j=1,2,3
\end{equation}
For the computations, we take $T_l=l \mbox{ for } l=0,1,...,$, $m_1=m_2=1$,  $k_2=1$,
\begin{equation}
k_1^{(l)}= \left\{\begin{array}{ll}
4, & l=0,2,4,... \\ 
9, & l=1,3,5,...  
\end{array} \right. 
\mbox{ and }
k_3^{(l)}= \left\{\begin{array}{ll}
9, & l=0,2,4,... \\ 
16, & l=1,3,5,...  
\end{array} \right. 
\end{equation}  
For this choice of values exact eigenvalues  of (\ref{eq:hco3-matrix}) are
\begin{equation}\label{eq:hco3-freqs}
\pm i\omega_{1,2}^{(l)}= \left\{\begin{array}{ll}
\pm i \sqrt{(27\pm\sqrt{53})/2}, & l=0,2,4,... \\ 
\pm i \sqrt{(15\pm\sqrt{29})/2}, & l=1,3,5,...  
\end{array} \right. 
\end{equation}  
For the exact solution (Fig.\ref{fig:hco3-01}) we use (\ref{eq:LNA-hybrid-fmatrix}).
 
In Fig.\ref{fig:cco3-01}(a) solution pairs $(x_1,x_3)$ and $(x_2,x_4)$ are plotted. 
In Fig.\ref{fig:hco3-01}(b) eigenvalues off the unit circle appear, which is consistent with the fact that due to elasticity coefficients switching, also the frequencies of coupled oscillators switch (\ref{eq:hco3-freqs}). Therefore, some instabilities appear, however are short lived and appears seemingly at random times. 

Eigenvalues of the dynamical system matrix do not reveal that instability in an obvious way (Fig.\ref{fig:hco3-01}), but the time-dependent Koopman eigenvalues clearly show both: the change in frequency of the oscillations (Fig.\ref{fig:hco3-03}(b)), and the bursts of the amplitude of the oscillations (Fig.\ref{fig:hco3-03}(a)). Even in this four dimensional dynamical system case with strongly coupled state variables, Algorithm 1 gives highly accurate results.

\begin{figure}[!htb]
	\centering
	\includegraphics[width=0.9\linewidth]{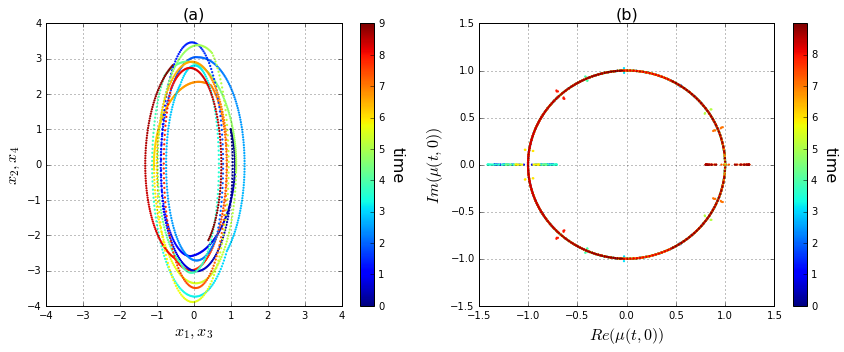}
	\caption{Exact solution starting with the inital condition (1,1,1,1) in the state space (a), and exact fundamental matrix eigenvalues (b) for the coupled oscillators with switching frequency (\ref{eq:hco3-matrix}).}
	\label{fig:hco3-01}
\end{figure}

\begin{figure}[!htb]
	\centering
	\includegraphics[width=0.9\linewidth]{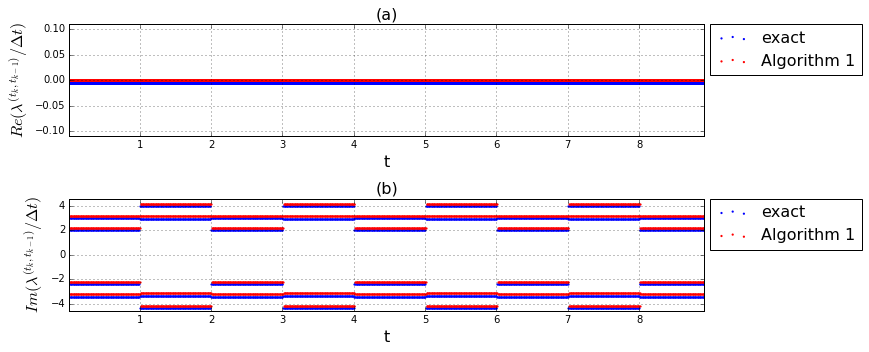}
	\caption{Dynamical system matrix eigenvalues for the coupled oscillators with switching frequency (\ref{eq:hco3-matrix}), computed with  Algorithm 1. (Exact values are offset to improve visibility.)}
	\label{fig:hco3-02}

	\includegraphics[width=0.9\linewidth]{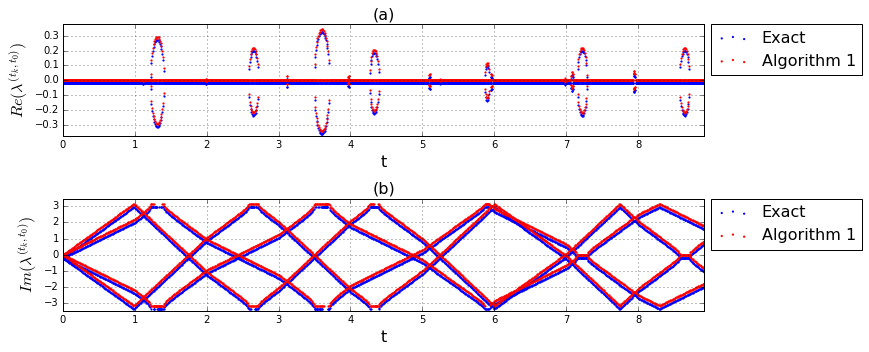}
	\caption{Koopman operator eigenvalues for the coupled oscillators with switching frequency (\ref{eq:hco3-matrix}), computed with  Algorithm 1. (Exact values are offset to improve visibility.)}
	\label{fig:hco3-03}
\end{figure}

\subsection{Multicompartment model with delay}\label{ss:mcm}

Multicompartment models are often used in medicine and pharmacoterapy.  These are models of the form 
\begin{equation}\label{eq:mcm_ode}
\dot{x}_{i} = - \sum_{j=1, j\ne i}^{n} k_{ij} x_i +  \sum_{j=1, j\ne i}^{n} k_{ji} x_j \quad \mbox{ for }i=1,...,n.
\end{equation}
Here $x_i=x_i(t)$ is the concentration of a substance in the $i^{th}$ compartment, $i=1,...,n$, and $k_{ij}$, is the rate coefficient of the transport of the substance from the $i^{th}$ to the $j^{th}$ compartment, $i,j=1,...,n$, $i\ne j$. 
As usual, concentrations are expressed as relative, i.e. as fractions of the sum of the initial concentrations. In the closed case (\ref{eq:mcm_ode}) the sum of the concentrations is constant, i.e. 
\begin{equation}\label{eq:mcm_sum}
\sum_{i=1}^{n}{x}_{i} = 1.
\end{equation}

If the transfer between two compartments starts only after some delay time, this can be modeled by using time-dependent coefficients of the form
\begin{equation}\label{eq:mcm_k}
k_{ij}(t)=\left\{
\begin{array}{ll}
0      &\mbox{ if } t<T_{ij}\\
K_{ij} &\mbox{ if } t\ge T_{ij} 
\end{array}
\right.
\end{equation}
where $T_{ij}$ is the delay time $i,j=1,...,n$, $i\ne j$. In such case the model is a hybrid linear non-autonomous system. 

We consider the multicompartment model for endosomal trafficking of $eL^d$ molecules \cite{mahmutefendic2017late}. In  that paper among other results, an application of a five compartment model was presented, where the non-zero rate coefficients and delay times were obtained that minimized the difference between measurements and simulation (Table \ref{tbl:mcm_kt}).

\begin{table}[!htb]
\centering
\caption{Example \ref{ss:mcm}, rate coefficients and delay times}
\label{tbl:mcm_kt}
\begin{tabular}{|c|c|c|}
	\hline
	(i,j) & $K_{ij}$ & $T_{ij}$ \\
	\hline
	(1,2)    & 0.0988 & 0 \\
	(2,1)    & 0.1410 & 5 \\
	(2,3)    & 0.0590 & 3 \\
	(3,4)    & 0.1150 & 18 \\
	(4,1)    & 0.0149 & 30 \\
	(4,5)    & 0.0154 & 55 \\
	\hline
\end{tabular}
\end{table}

In the computations we must take care of the fact that the number of independent observables is not equal to the state dimension for two reasons. The first reason is (\ref{eq:mcm_sum}) so concentration in the $5^{th}$ compartment can be eliminated from the observable set. The second reason is that zero values of the rate coefficients may cause that some compartments are completely inactive for some time. So again, those concentrations should not be included in the set of observables. The identification of the necessary number of observables can be easily achieved by checking the dimension of the Krylov subspace. 

With this addition to the Algorithm 1, we obtain excellent results as it can be seen in Fig.\ref{fig:mcm-02} and Fig.\ref{fig:mcm-03}. The eigenvalues of the underlying matrix and the time delays are correctly identified (Fig.\ref{fig:mcm-02}). In Fig.\ref{fig:mcm-03}(b) we see that the imaginary part of the time-dependent Koopman eigenvalue is zero, so there are no oscillations in this system. In Fig.\ref{fig:mcm-03}(a) real part of one of the Koopman eigenvalues is zero, which is consistent with the fact that the sum of all concentrations is constant in time (\ref{eq:mcm_sum}). Other time-dependent real parts of Koopman eigenvalues are negative and decreasing, so related particular solutions of the governing equations are vanishing as time increases. This is consistent with the behavior of the solution, since it is obviously converging to a steady state (Fig.\ref{fig:mcm-01}).

\begin{figure}[!htb]
	\centering
	\includegraphics[width=0.9\linewidth]{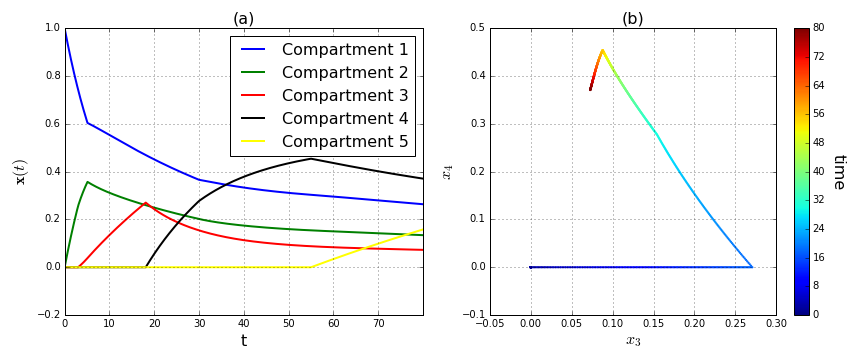}
	\caption{Exact solution starting with the inital condition (1,0,0,0,0) (a), and exact compartment concentrations in state space (b) for the multicompartment model with delay (\ref{eq:mcm_ode}).}
	\label{fig:mcm-01}
\end{figure}

\begin{figure}[!htb]
	\centering
	\includegraphics[width=0.9\linewidth]{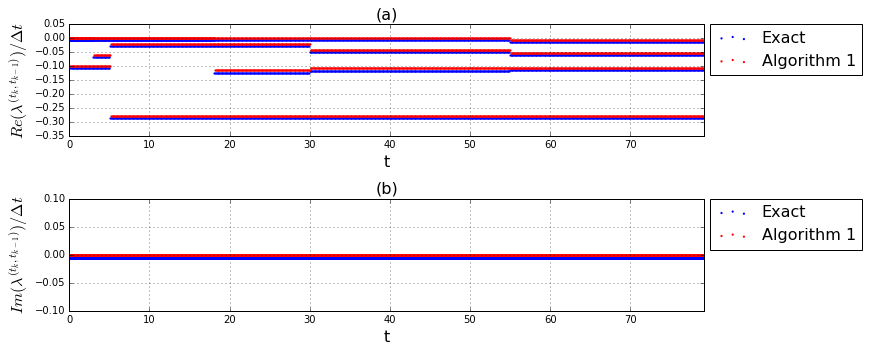}
	\caption{Dynamical system matrix eigenvalues for the multicompartment model with delay (\ref{eq:mcm_ode}), computed with Algorithm 1. (Exact values are offset to improve visibility.)}
	\label{fig:mcm-02}

	\includegraphics[width=0.9\linewidth]{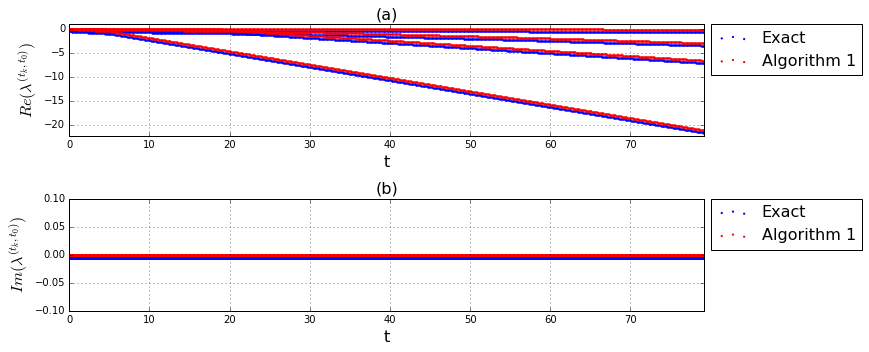}
	\caption{Koopman operator eigenvalues for the multicompartment model with delay (\ref{eq:mcm_ode}), computed with Algorithm 1. (Exact values are offset to improve visibility.)}
	\label{fig:mcm-03}
\end{figure}

\subsection{Continuous frequency change}\label{ss:c2f}

In this example we consider an oscillator with continously changing frequency. The governing equations are (\ref{eq:LNA-system}) with the underlying matrix (\ref{eq:LNA-oscillation}) where we additionally set $\sigma(t)=0$ and 
\begin{equation}\label{eq:c2f}
\omega(t)=\omega_0+A_d\cos(\omega_d t)+B_d\sin(\omega_d t)
\end{equation}
We can solve it analytically using fundamental matrix (\ref{eq:osci-fmatrix}) with $\alpha(t,t_0)=0$ and
$$ \beta(t,t_0)=\omega_0(t-t_0)+ 
 \frac{A_d}{\omega_d}\left(\sin(\omega_d t) - \sin(\omega_d t_0)\right) -
   \frac{B_d}{\omega_d}\left(\cos(\omega_d t) - \cos(\omega_d t_0)\right)$$

The computations are performed for $\omega_0=2$, $\omega_d=\pi$, and $A_d=0.5$.

Exact solution (Fig.\ref{fig:c2f-01}) shows that this continuous change in frequency of the underlying matrix does not produce instabilities. All eigenvalues are on the unit circle and there are no damping-driving effects.
This is confirmed by the results obtained using Algorithm 2. In Fig.\ref{fig:c2f-02}(a) and Fig.\ref{fig:c2f-03}(a) we see that real parts of both dynamical system matrix and Koopman operator eigenvalues stay equal to zero at all times. Imaginary parts of dynamical system matrix eigenvalues (Fig.\ref{fig:c2f-02}(b)) and Koopman operator eigenvalues (Fig.\ref{fig:c2f-03}(b)) computed with Algorithm 2 show the correct time-dependency. 

However, when computations are performed with any Arnoldi-like method on moving stencils, numerically evaluated real part of the eigenvalue of the underlying matrix is displaying a nonexistent time-dependency (Fig.\ref{fig:c2f-02-err}(a)), while the imaginary parts of those eigenvalues are correct (Fig.\ref{fig:c2f-02-err}(a)). As proven in Theorem \ref{thm:NLA-Aerror} the numerical result for the imaginary part of the dynamical system matrix eigenvalues are correct because there is no time change in $\sigma$. Also, as proven in  Theorem \ref{thm:NLA-Aerror} numerical results for the real part of the dynamical system matrix eigenvalues are compromised by the error which is proportional to the time derivative of $\omega(t)$. 

This error then propagates into the Koopman operator eigenvalue computations (Fig. \ref{fig:c2f-03-err}(a) and (b)). From those results it might be concluded that there is an amplitude change in the system, and that even at some time moments frequencies stay at the value $\pm \pi$ (Fig.\ref{fig:c2f-03-err}(b)) and then real parts of eigenvalues split into two different values (Fig.\ref{fig:c2f-03-err}(a)). All of this is completely erroneous, as it is additionally confirmed by the algorithm error plotted in Fig.\ref{fig:c2f-03-err}(c).

\begin{figure}[!htb]
	\centering
	\includegraphics[width=0.9\linewidth]{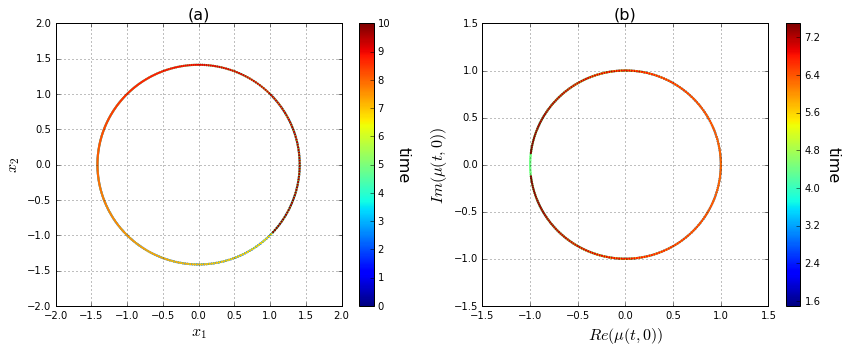}
	\caption{Exact solution starting with the inital condition (1,1) in the state space (a), and exact fundamental matrix eigenvalues (b) for the dynamical system with continuous frequency change (\ref{eq:c2f}).}
	\label{fig:c2f-01}
\end{figure}

\begin{figure}[!htb]
	\centering
	\includegraphics[width=0.9\linewidth]{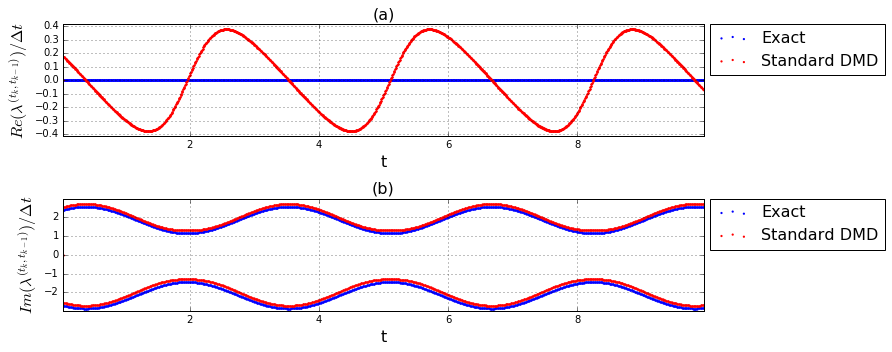}
	\caption{Dynamical system matrix eigenvalues for the dynamical system with continuous frequency change (\ref{eq:c2f}), computed with standard DMD.}
	\label{fig:c2f-02-err}

	\includegraphics[width=0.9\linewidth]{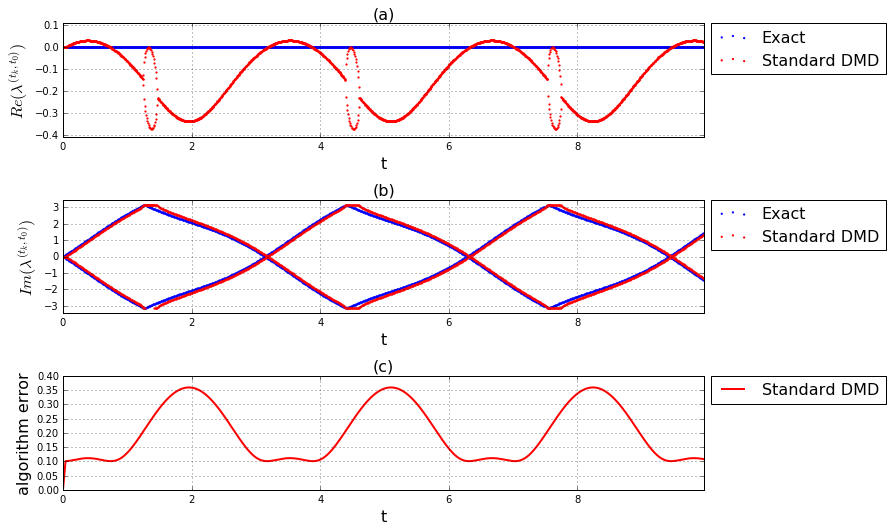}
	\caption{Koopman operator eigenvalues  and the algorithm error (\ref{eq:data-alg-err}) for the dynamical system with continuous frequency change (\ref{eq:c2f}), computed with standard DMD.}
	\label{fig:c2f-03-err}
\end{figure}

\begin{figure}[!htb]
	\centering
	\includegraphics[width=0.9\linewidth]{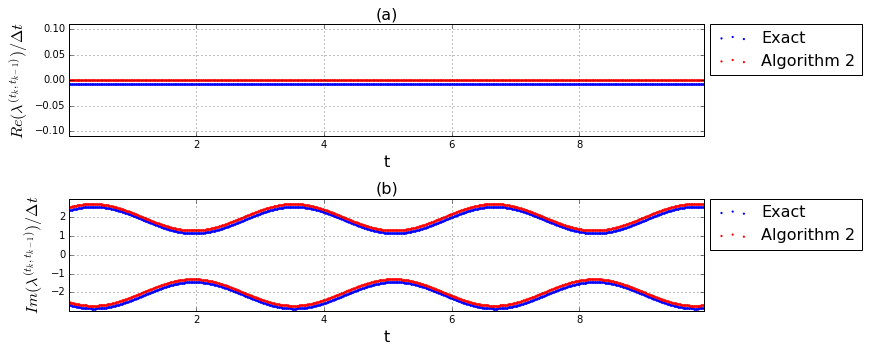}
	\caption{Dynamical system matrix eigenvalues for the dynamical system with continuous frequency change (\ref{eq:c2f}), computed with Algorithm 2. (Exact values are offset to improve visibility.)}
	\label{fig:c2f-02}

	\includegraphics[width=0.9\linewidth]{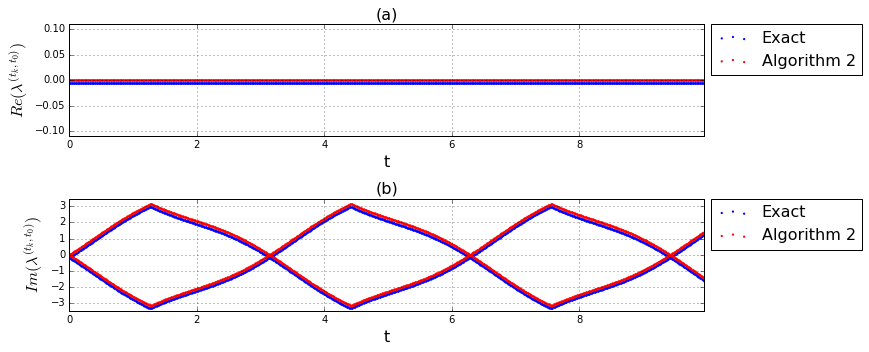}
	\caption{Koopman operator eigenvalues for the dynamical system with continuous frequency change (\ref{eq:c2f}), computed with Algorithm 2. (Exact values are offset to improve visibility.)}
	\label{fig:c2f-03}
\end{figure}

\subsection{Continuous change in damping rate}\label{ss:cdd}

Here we consider an oscillator with a continously changing amplitude. It is a linear non-autonomous system (\ref{eq:LNA-system}) with underlying matrix (\ref{eq:LNA-oscillation}) where we additionally set $\omega(t)=\omega_0$ and 
\begin{equation}\label{eq:cdd}
 \sigma(t)=\sigma_0+A_d\cos(\omega_d t)+B_d\sin(\omega_d t) 
\end{equation}
We can solve it analytically using fundamental matrix (\ref{eq:osci-fmatrix}) with $\beta(t,t_0)=0$ and
$$ \alpha(t,t_0)=\sigma_0(t-t_0)+ 
 \frac{A_d}{\omega_d}\left(\sin(\omega_d t) - \sin(\omega_d t_0)\right) -
\frac{B_d}{\omega_d}\left(\cos(\omega_d t) - \cos(\omega_d t_0)\right)$$

The computations are performed for  $\sigma_0=0$, $\omega_0=2$, $\omega_d=\pi$, and $A_d=0.5$. 
Since amplitude is changing the solution appears non-symmetrical relative to axis in the state space (Fig.\ref{fig:cdd-01}(a)), and this change is confirmed in Fig.\ref{fig:cdd-01}(b). 

If the computations on snapshots are performed with any Arnoldi-like method on moving stencils we get results as in Fig.\ref{fig:cdd-02-err} and Fig.\ref{fig:cdd-03-err}. As proven in Theorem \ref{thm:NLA-Aerror} since $\omega$ is constant in time, there is no error in numerical real parts of the dynamical system matrix eigenvalues (Fig.\ref{fig:cdd-02-err}(a)). Also, the error in the numerical imaginary part of the dynamical system matrix eigenvalues is, as explained in the same theorem, proportional to the time derivative of $\sigma(t)$ (Fig.\ref{fig:cdd-02-err}(b)). This causes error in numerical time-dependent Koopman operator eigenvalues (Fig.\ref{fig:cdd-03-err}(a) and (b)), and produces large algorithm error (Fig.\ref{fig:cdd-03-err}(c)).

All these issues can be eliminated by applying Algorithm 2, which gives highly accurate results, as we show in Fig.\ref{fig:cdd-02} and Fig.\ref{fig:cdd-03}.

\begin{figure}[!htb]
	\centering
	\includegraphics[width=0.9\linewidth]{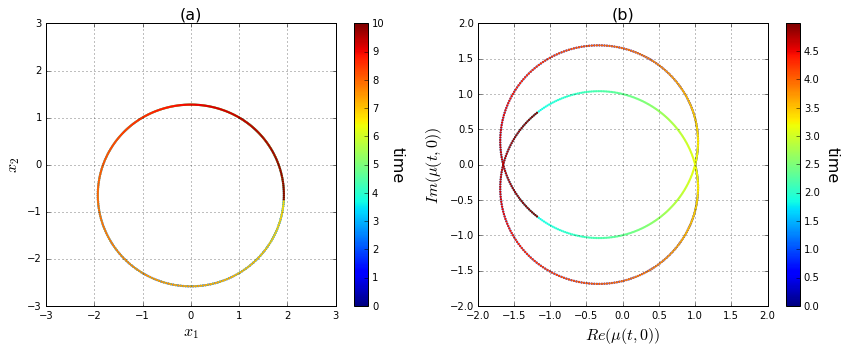}
	\caption{Exact solution starting with the inital condition (1,1) in the state space (a), and exact fundamental matrix eigenvalues (b) for the dynamical system with continuous change in damping rate (\ref{eq:cdd}).}
	\label{fig:cdd-01}
\end{figure}

\begin{figure}[!htb]
	\centering
	\includegraphics[width=0.9\linewidth]{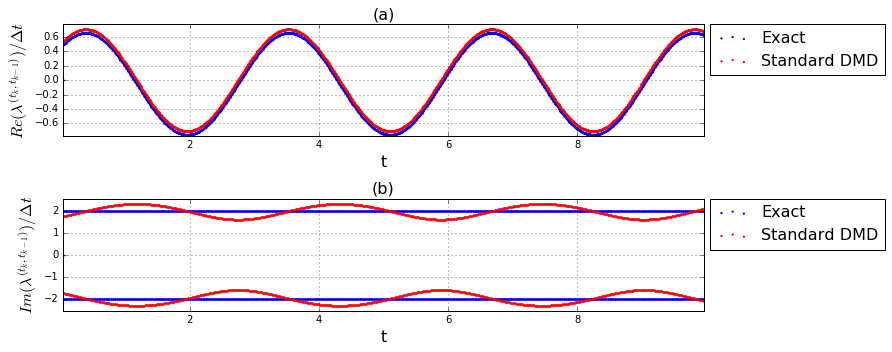}
	\caption{dynamical system matrix eigenvalues for the dynamical system with continuous change in damping rate (\ref{eq:cdd}), computed with standard DMD.}
	\label{fig:cdd-02-err}

	\includegraphics[width=0.9\linewidth]{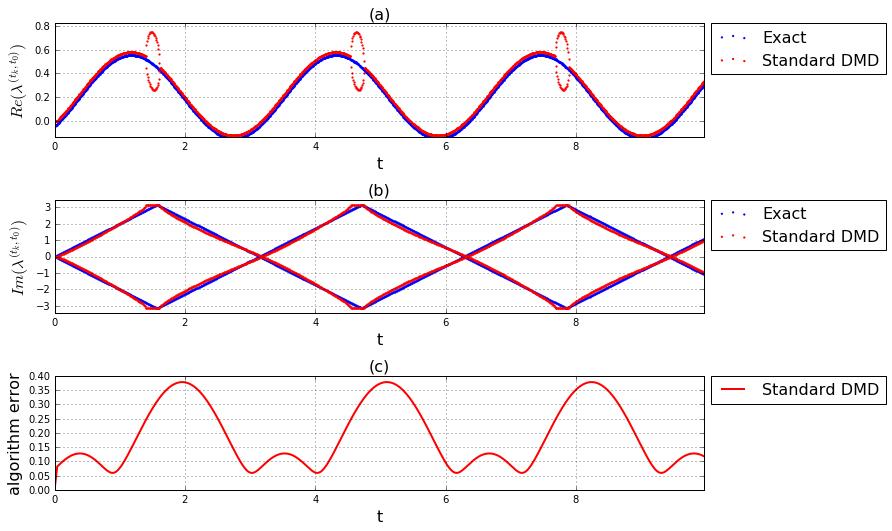}
	\caption{Koopman operator eigenvalues  and the algorithm error (\ref{eq:data-alg-err}) for the dynamical system with continuous change in damping rate (\ref{eq:cdd}), computed with standard DMD.}
	\label{fig:cdd-03-err}
\end{figure}

\begin{figure}[!htb]
	\centering
	\includegraphics[width=0.9\linewidth]{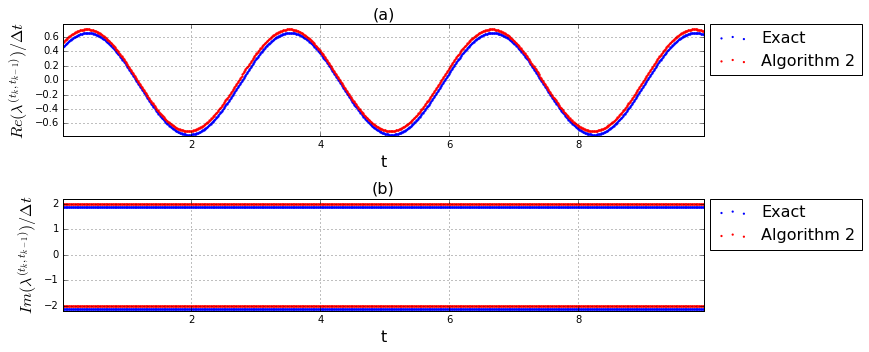}
	\caption{Dynamical system matrix eigenvalues for the dynamical system with continuous change in damping rate (\ref{eq:cdd}), computed with Algorithm 2. (Exact values are offset to improve visibility.)}
	\label{fig:cdd-02}

	\includegraphics[width=0.9\linewidth]{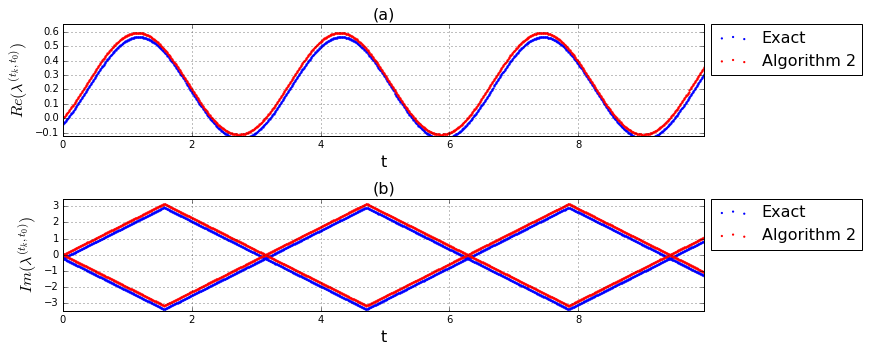}
	\caption{Koopman operator eigenvalues for the dynamical system with continuous change in damping rate (\ref{eq:cdd}), computed with Algorithm 2. (Exact values are offset to improve visibility.)}
	\label{fig:cdd-03}
\end{figure}

\subsection{Nonautonomous coupled oscillators}\label{ss:cco3}

Let us consider two coupled oscillators (similar to Section \ref{ss:hco3}), but now both with continously changing frequencies. 
In order to solve such a system analytically we write it in the equivalent form (see \cite{Veselic})
with underlying matrix
\begin{equation}\label{eq:cco3-matrix}
\mathbf{A}(t) = 
\left(\begin{array}{cccc}
0 & 0 & \omega_1(t) & 0 \\ 
0 & 0 & 0 & \omega_2(t) \\ 
-\omega_1(t) & 0 & 0 & 0 \\ 
0 & -\omega_2(t) & 0 & 0  
\end{array} \right).   
\end{equation}
Here $\omega_j(t)=\sqrt{\nu_j(t)}$, $j=1,2$, and $\nu_j(t)$, $j=1,2$ are computed from the generalized eigenvalue problem (\ref{eq:co3-geneig-1})-(\ref{eq:co3-geneig-2}) for each $t>t_0$.

This form is commutative, so we can compute the fundamental matrix

\begin{equation}\label{eq:cco3-fmatrix}
\mathcal{M}(t,t_0)=
\left(\begin{array}{cccc}
 \cos\beta_1 & 0 & \sin\beta_1 & 0 \\ 
 0 & \cos\beta_2 & 0 & \sin\beta_2 \\ 
-\sin\beta_1 & 0 & \cos\beta_1 & 0 \\
0 & -\sin\beta_2 & 0 & \cos\beta_2 
\end{array} \right),   
\end{equation}
where
\begin{equation}\label{eq:cco3-ab}
\beta_j=\beta_j(t,t_0)=\int_{t_0}^{t}\omega_j(\tau)d\tau, j=1,2.
\end{equation}  

This is a four dimensional test example with all variables strongly coupled. 

In the computations we take $m_1=m_2=1$, $k_1=2$, $k_2=1$, $k_3=3$, and we compute constant part of the frequency $\omega_{1,2}^{(0)}=\sqrt{(7\pm\sqrt{5})/2}$, from (\ref{eq:co3-geneig-1})-(\ref{eq:co3-geneig-2}). Then we add a frequency forcing term 
\begin{equation}\label{eq:cco3-fast}
\omega_1(t)=\omega_1^{(0)}+\frac{1}{2}\left(\cos(2 t)+\sin(2 t)\right)
\end{equation}
to one oscillator, and another frequency forcing term
\begin{equation}\label{eq:cco3-slow}
\omega_2(t)=\omega_2^{(0)}+\frac{1}{2}\left(\cos(0.4 t)+\sin(0.4 t)\right)
\end{equation}
to the other oscillator. Obeserve that from the perspective of the Theorem \ref{thm:NLA-Aerror} (\ref{eq:cco3-fast}) will produce larger error then (\ref{eq:cco3-slow}).

In Fig.\ref{fig:cco3-01}(a) solution pairs $(x_1,x_3)$ and $(x_2,x_4)$ are plotted. 

This example is important since it is higher dimensional and the underlying matrix is not of the form examined in Theorem \ref{thm:NLA-Aerror}. We want to see if something similar to what is proven in that theorem will occur, and test if Algorithm 2 with appropriate choice of observables gives good results.

First, we compute the dynamical system matrix and Koopman operator eigenvalues with Standard DMD on moving stencils and obtain results presented in Fig.\ref{fig:cco3-02-err} and Fig.\ref{fig:cco3-03-err}. As expected, the exact real part of the eigenvalues should be zero, but the numerical real parts of eigenvalues exhibit an error (Fig.\ref{fig:cco3-02-err}(a)). 
The only part that is almost accurately computed is the part related to the the slower frequency forcing term (\ref{eq:cco3-slow}) (Fig.\ref{fig:cco3-02-err}(b)).But what we observe in Fig.\ref{fig:cco3-02-err}(b) is that errors contributed to imaginary eigenvalues are not only by real parts (that are zero here) but also by imaginary part of the other oscillator.
This indicates, that the error proven in Theorem \ref{thm:NLA-Aerror}, escalates with the complexity of the system, which is quite logical. 

Now in order to apply Algorithm 2, in step 2 we define new observables 
\begin{equation}\label{eq:cco3-good-obs-1}
u_1 = \sqrt{x_1^2+x_3^2}, 
u_2 = (x_1+ix_3)/u_1, 
\end{equation} 
\begin{equation}\label{eq:cco3-good-obs-2}
u_3 = \sqrt{x_2^2+x_4^2}, 
u_4 = (x_2+ix_4)/u_4, 
\end{equation} 
Notice that the observable pair $(u_1,u_2)$ is obtained with a nonlinear conjugate transformation of form (\ref{eq:ConjMap}) on state variables $(x_1,x_3)$, and the same is valid for observable pair $(u_3,u_4)$ and state variables $(x_2,x_4)$. 
Comparison between exact eigenvalues and numerical eigenvalues obtained with Algorithm 2 (Fig.\ref{fig:cco3-02} and Fig.\ref{fig:cco3-03}), one more time shows that the issues disappear and that Algorithm 2 gives highly accurate results.

\begin{figure}[!htb]
	\centering
	\includegraphics[width=0.9\linewidth]{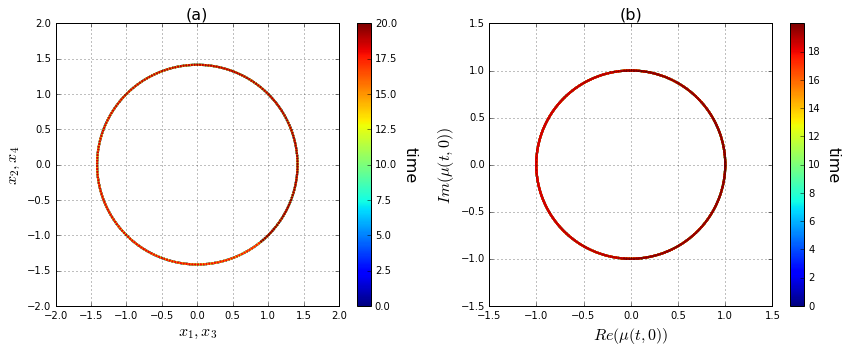}
	\caption{Exact solution starting with the inital condition (1,1,1,1) in the state space (a), and exact fundamental matrix eigenvalues (b) for the coupled oscillators with continuous frequency change (\ref{eq:cco3-matrix}).}
	\label{fig:cco3-01}
\end{figure}

\begin{figure}[!htb]
	\centering
	\includegraphics[width=0.9\linewidth]{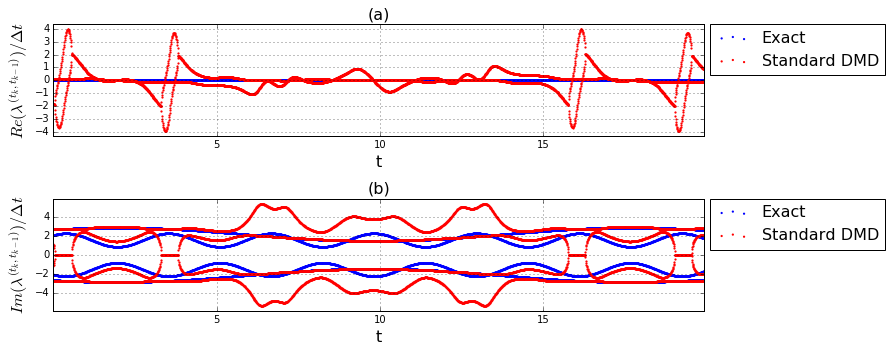}
	\caption{Dynamical system matrix eigenvalues for the coupled oscillators with continuous frequency change (\ref{eq:cco3-matrix}), computed with standard DMD.}
	\label{fig:cco3-02-err}

	\includegraphics[width=0.9\linewidth]{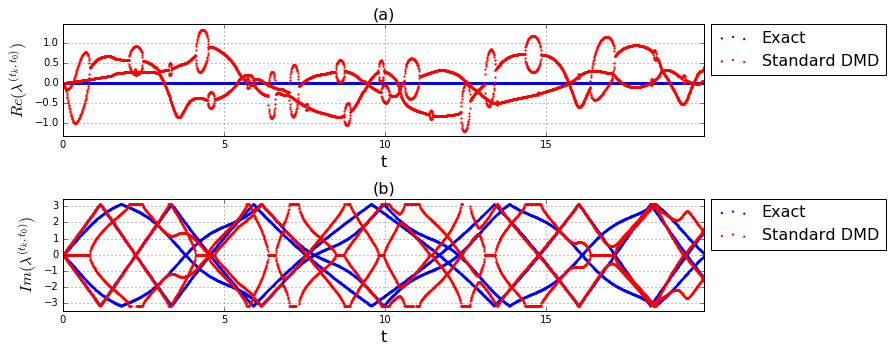}
	\caption{Koopman operator eigenvalues for the coupled oscillators with continuous frequency change (\ref{eq:cco3-matrix}), computed with standard DMD.}
	\label{fig:cco3-03-err}
\end{figure}

\begin{figure}[!htb]
	\centering
	\includegraphics[width=0.9\linewidth]{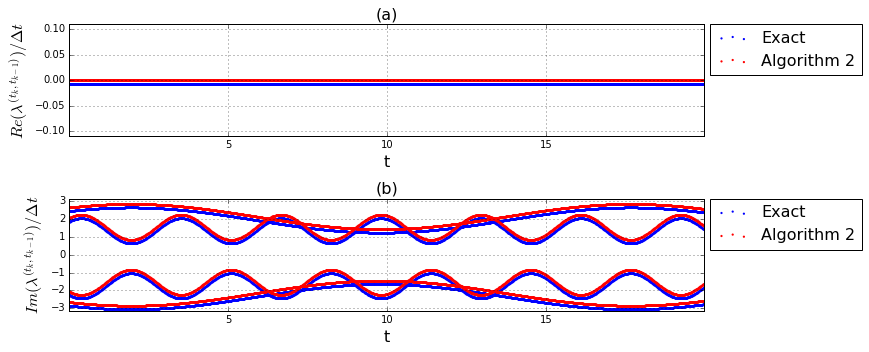}
	\caption{Dynamical system matrix eigenvalues for the coupled oscillators with continuous frequency change (\ref{eq:cco3-matrix}), computed with Algorithm 2. (Exact values are offset to improve visibility.)}
	\label{fig:cco3-02}

	\includegraphics[width=0.9\linewidth]{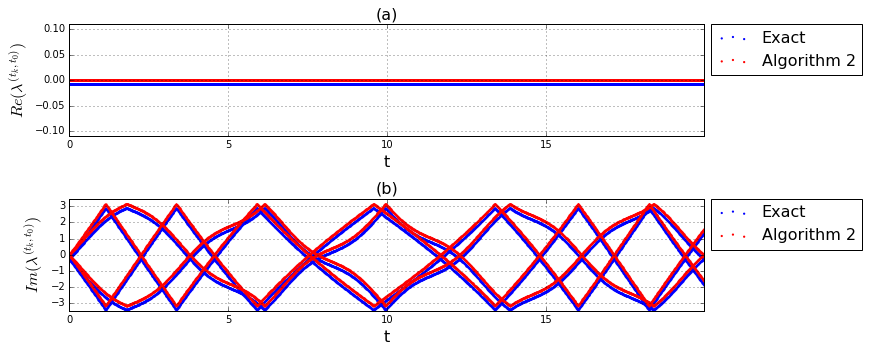}
	\caption{Koopman operator eigenvalues for the coupled oscillators with continuous frequency change (\ref{eq:cco3-matrix}), computed with Algorithm 2. (Exact values are offset to improve visibility.)}
	\label{fig:cco3-03}
\end{figure}

\section{Conclusion}

In this paper we prove a close connection between Koopman operator family for linear non-autonomous dynamical systems and the fundamental matrix family for the underlying system of governing equations. This results in a finite dimensional Koopman expansion for the full state observable.
Actually, if a new set of observables is introduced, which after transformation also satisfy governing equations of form (\ref{eq:LNA-system}) again the same Theorem \ref{thm:kopp-fmatrix} can be applied.

Then we analyze data-driven algorithms on hybrid linear non-autonomous systems. What we discover is that the increase in Krylov subspace projection error signals the time moments, when switching of the values in the underlying matrix occurs. An appropriate use of this information leads us to a full detection of the governing equations. This approach is formalized as the Algorithm 1.

Another important result is the revealed nature of the error that occurs when Arnoldi-like methods are used for approximation of time-dependent underlying matrices. The error in these approximations is proportional to time derivatives of the eigenvalues. It is obvious that at the core of this issue lies the fact that Arnoldi-like methods are constructed to capture eigenvalues of constant high-dimensional matrices. Particularly significant is the fact that this error does not vanish if we decrease the time step between snapshots.  

In order to solve this issue for the continuously time-dependent nonautonomous systems we propose Algorithm 2. The essence of this algorithm is that we must introduce observables containing only one real or one imaginary part of the eigenvalues, so that stencil with only two snapshots contains enough information to reconstruct that value. Arnoldi-like methods that expect the underlying matrix is constant, will, on any larger stencil span, produce significant error. One good path for the definition of appropriate observables is through nonlinear conjugate transformations of the form (\ref{eq:ConjMap}) which might lead to discovery of new, adaptive algorithms for observable selection.

All numerical results show that application of Arnoldi-like methods on moving stencils to data collected on non-autonomous systems, will lead to wrong conclusions about the nature of the dynamical system. Also, all numerical tests confirm high accuracy of both algorithms that are proposed in this paper.

\section{Acknowledgment}

This research has been supported by the DARPA Contract HR0011-16-C-0116 "On A Data-Driven,
Operator-Theoretic Framework for Space-Time Analysis of Process Dynamics". S.M. is grateful to Dr. Ryan Mohr and Dr. Maria Fonoberova for helpful mathematical discussions and comments on the manuscript.

\bibliographystyle{plain}
\bibliography{bibliography}

\end{document}